\documentclass[a4paper,12pt]{article}
\usepackage{amsmath,amsthm}
\usepackage{amssymb}
\usepackage{latexsym}
\usepackage{graphicx}
\usepackage{tikz}
\usepackage{scalefnt}
\usepackage{lmodern}
\usepackage{float}
\usetikzlibrary{positioning}
\usetikzlibrary{shapes,arrows}
\usepackage{subfigure}
\usepackage{pgfplots}
\usepackage{MnSymbol}
\usepackage{wrapfig}
\pgfplotsset{compat=1.7}
\usepackage[pagestyles]{titlesec}
\usepackage[pdfpagemode=UseNone,pdfstartview=FitH,hypertexnames=false]{hyperref}
\usepackage{cases}
\usepackage[margin=1in]{geometry}
\newtheorem{theorem}{Theorem}[section]
\newtheorem{corollary}{Corollary}[section]
\newtheorem{remark}{Remark}[section]
\newtheorem{example}{Example}[section]
\newtheorem{lemma}{Lemma}[section]

\newtheorem{definition}{Definition}[section]
\newtheorem{proposition}{Proposition}[section]
\newtheorem{problem}{Problem}[section]
  
\newcommand*{\Zset}{\mathbb{Z}}

\begin{document}
\title{{\Large \bf Studying the inertias of LCM matrices and revisiting the Bourque-Ligh conjecture}}
\author{{\sc Mika Mattila\footnote{Corresponding author, who was supported by the Finnish Cultural Foundation, Pirkanmaa Regional Fund.}, Pentti Haukkanen and Jori Mäntysalo}
\\Faculty of Information Technology and Communication Sciences,\\Tampere University, Finland
\\E-mail:  mika.mattila@tuni.fi, pentti.haukkanen@tuni.fi,\\ jori.mantysalo@tuni.fi}
\maketitle
\setcounter{section}{0}
\setcounter{equation}{0}

\abstract{Let $S=\{x_1,x_2,\ldots,x_n\}$ be a finite set of distinct positive integers. Throughout this article we assume that the set $S$ is GCD closed. The LCM matrix $[S]$ of the set $S$ is defined to be the $n\times n$ matrix with $\mathrm{lcm}(x_i,x_j)$  as its $ij$ element.  The famous Bourque-Ligh conjecture used to state that the LCM matrix of a GCD closed set $S$ is always invertible, but currently it is a well-known fact that any nontrivial LCM matrix is indefinite and under the right circumstances it can be even singular (even if the set $S$ is assumed to be GCD closed). However, not much more is known about the inertia of LCM matrices in general. The ultimate goal of this article is to improve this situation. Assuming that $S$ is a meet closed set we define an entirely new lattice-theoretic concept by saying that an element $x_i\in S$ generates a double-chain set in $S$ if the set $\mathrm{meetcl}(C_S(x_i))\setminus C_S(x_i)$ can be expressed as a union of two disjoint chains (here the set $C_S(x_i)$ consists of all the elements of the set $S$ that are covered by $x_i$ and $\mathrm{meetcl}(C_S(x_i))$ is the smallest meet closed subset of $S$ that contains the set $C_S(x_i)$). We then proceed by studying the values of the Möbius function on sets in which every element generates a double-chain set and use the properties of the Möbius function to explain why the Bourque-Ligh conjecture holds in so many cases and fails in certain very specific instances. After that we turn our attention to the inertia and see that in some cases it is possible to determine the inertia of an LCM matrix simply by looking at the lattice-theoretic structure of $(S,|)$ alone. Finally, we are going to show how to construct LCM matrices in which the majority of the eigenvalues is either negative or positive.}


\section{Introduction}

LCM matrices, as well as GCD matrices, were first defined by H. J. S. Smith \cite{Smi} in his seminal paper from the year 1876. By letting $S=\{x_1,x_2,\ldots,x_n\}$ be a finite set of distinct positive integers, Smith defined the GCD matrix $(S)$ of the set $S$ to be the $n\times n$ matrix with $\gcd(x_i,x_j)$ as its $ij$ element. Similarly, the LCM matrix $[S]$ of the set $S$ is the $n\times n$ matrix with $\mathrm{lcm}(x_i,x_j)$ as its $ij$ element.  Although in these days Smith's paper is probably best remembered from its famous determinant formula for the GCD matrix with $\gcd(i,j)$ as its $ij$ element, Smith also considered the determinants of more general GCD and LCM matrices. For example, he showed that if the set $S$ is \emph{factor closed}, then both of the matrices $(S)$ and $[S]$ are invertible (the set $S$ is said to be factor closed if the condition
\[
y\,|\,x\quad\text{for\ some\ }x\in S\Rightarrow y\in S
\]
holds). Since Smith, many other authors have also considered the determinants of GCD-related matrices (see the references in \cite{HauWanSil}). In 1989 Beslin and Ligh \cite{BesLigh89} reintroduced GCD matrices and also initiated a series of articles studying GCD-type matrices and their generalizations. However, LCM matrices did not get much attention until the article \cite{Bour92} by Bourque and Ligh appeared. Among other things, in this article it is pointed out how easy it is to find singular LCM matrices by considering the LCM matrix of the set $S=\{1,2,15,42\}$ (see \cite[p. 68]{Bour92}). In that same paper the authors follow in the footsteps of Smith and are interested in finding a more general sufficient condition (comparing to the factor-closedness of the set $S$) for the invertibility of the LCM matrix $[S]$. They ended up conjecturing that the GCD-closedness of the set $S$ suffices to guarantee the invertibility of $[S]$.

In 1997 Haukkanen et al. \cite{HauWanSil} were able disprove the Bourque-Ligh conjecture by finding a singular LCM matrix of size $9\times 9$. Two years later Hong \cite{Hong99} was able to find another counterexample of size $8\times 8$. By using number-theoretic methods he also showed that the conjecture holds for GCD closed sets with at most $7$ elements, and thus the conjecture was solved completely (at least in some sense). However, since in the conjecture it is assumed that the set $S$ is GCD closed, the structure $(S,|)$ itself constitutes a meet-semilattice, which enables one to study the conjecture from an entirely lattice-theoretic point of view. In \cite{KorMatHau} Korkee et al. consider all possible semilattice structures with at most $7$ elements and showed that the LCM matrix $[S]$ is invertible for any GCD closed set $S$ with $|S|\leq7$. In \cite{MatHau3} this same lattice-theoretic approach is utilized to show that if the matrix $[S]$ is singular and the set $S$ is GCD closed with $8$ elements, then $(S,|)$ has unique, cube-like structure (see Figure \ref{fig:ex1} (e)). These same methods were also adapted by Altinisik et al. in \cite{Alt17}, where they study the singularity of the matrix $[S]$ in the case when $S$ is a GCD closed set with $9$ elements. At the same time the invertibility of the so-called power GCD and power LCM matrices have been studied from a number-theoretic point of view by several different authors, see e.g. \cite{Hong04, Li, Shen}.

There are several articles in which GCD-type matrices have been studied by using lattice-theoretic methods, see e.g. \cite{Alt05, HongSun, IlmKaar, Ovall}. Our first goal here is to study the problems arising from our own previous lattice-theoretic studies of the Bourque-Ligh conjecture \cite{MatHau3} and \cite{KorMatHau}. Although there seems to be nothing left to prove in the Bourque-Ligh conjecture itself, there are a couple of interesting problems relating to the conjecture that remain open. Especially we are interested to find some answers to the following questions:

\begin{itemize}
	\item Since there are more than 1300 meet semilattices with at most 8 elements, the lattice-theoretic method basically requires one to sieve off most of the irrelevant cases by using some mathematical program (e.g. SageMath).	Is there any elegant way to avoid this?
	\item In many cases the lattice structure of $(S,|)$ alone suffices to guarantee the invertibility of $[S]$. What are the required semilattice properties that make the lattice-theoretic method to work?
	\item Cube semilattice with 8 elements is the smallest possible counterexample for Bourque-Ligh conjecture. What makes this structure so special?
\end{itemize}

After answering these questions we turn our attention to the inertia of LCM matrices (of GCD closed sets). Currently very little is known about this topic. So far all the existing inertia-related results have been presented in the articles \cite{HauToth, MatHau4, Ovall}. It turns out that the work that we have done to study the invertibility of LCM matrices can actually be directly applied to study the inertia of LCM matrices. In fact, it may be even a bit surprising that in many cases we are able to determine the inertia of a given LCM matrix by looking only at the semilattice structure of $(S,|)$. Of course this is not always possible, and for that reason we also study the limits of our method (i.e. the cases in which our approach gives an inconclusive result). Another thing for us to consider is the question about how to construct a GCD closed set $S$ such that the number of either negative or positive eigenvalues of the matrix $[S]$ is maximized. 

We begin our study in Section \ref{sect:ABset} by defining a certain double-chain generating property, an entirely new lattice-theoretic concept, and also by studying some of the basic properties of sets in which every element possesses this property. In Section \ref{sect:Möbius} we are then able to calculate the values of the Möbius function on a set in which every generates a double-chain set and thereby find out that the Möbius function values can easily be determined from the Hasse diagram of the set $(S,|)$ simply by calculating the number of certain type of elements in $S$. In Section \ref{sect:inv} we are finally ready to apply the Möbius inversion and to give a proof to one of the main theorems and to explain how the structure of $(S,|)$ often causes the matrix $[S]$ to be invertible. In Section \ref{sect:inertia} we take another look at the proof of the main theorem of Section \ref{sect:inv} and we see that in those previously mentioned cases the structure of $(S,|)$ not only guarantees the invertibility of the LCM matrix $[S]$ but also determines the inertia of this matrix completely.

\section{\texorpdfstring{$A$}{A}-sets and double-chain generating sets}\label{sect:ABset}

Let $(P,\preceq)$ be a locally finite meet semilattice. In other words, we assume that $\prec$ is a partial order relation on $P$, for all $x,y\in P$ the interval $[x,y]:=\{z\in P\,\big|\,x\preceq z\preceq y\}$ is a finite set and for any pair of elements $x,y\in P$ there exists a unique greatest common lower bound denoted by $x\wedge y$.  Moreover, let $S=\{x_1,x_2,\ldots,x_n\}$ be a meet closed subset of $P$ (i.e. $x_i\wedge x_j\in S$ for all $x_i,x_j\in S$) with $x_i\preceq x_j\Rightarrow i\leq j$, and let $C_S(x)$ denote the set of elements of $S$ which are covered by $x$ in $S$. In other words,
\[
C_S(x)=\{y\in S\,\big|\,y\prec x\text{\ and for all\ }z\in S:\ (y\preceq z\prec x\Rightarrow y=z)\}.
\]
If $y\in C_S(x)$, then we may also use the notation $y\lessdot x$. Throughout this article we are interested in the poset-theoretic structure of the meet closure of the set $C_S(x)$ defined as
\[
\mathrm{meetcl}(C_S(x)):=\{y_1\wedge y_2\wedge\cdots\wedge y_k\,\big|\,k\in\mathbb{Z}_+\text{\ and\ }y_1,y_2,\ldots,y_k\in C_S(x)\}.
\]

\begin{definition}\cite[Definition 2.2]{ISmo}\label{def:A-set}
The set $S=\{x_1,x_2,\ldots,x_n\}\subseteq P$ is an $A$-set if the set $A=\{x_k\wedge x_l\,\big|\,1\leq k<l\leq n\}$ is a chain.
\end{definition}

\begin{definition}
An element $x\in S\subseteq P$ generates a double-chain set in $S$ if the set $\mathrm{meetcl}(C_S(x))\setminus C_S(x)$ in $P$ can be expressed as a union of two disjoint sets $A$ and $B$ that are chains in $P$. 
\end{definition}

In the previous definition it would be more accurate to say that \emph{$x$ generates a double-chain set in $S$ via the set $C_S(x)$}, but for practical reasons it makes sense to use briefer terminology. Moreover, if the set $S$ is clear from the context we may abbreviate our notation even further and just say that $x$ generates a double-chain. In the case when every element $x\in S$ generates a double-chain set in $S$ we may also say that \emph{$S$ is a double-chain set}. 

\begin{remark}
It should be noted that although the chains $A$ and $B$ are disjoint, it is well possible that $a_k\prec b_l$ or $b_l\prec a_k$ for some $a_k\in A$ and $b_l\in B$. For example, the element $\bigwedge C_S(x)$ may be chosen to belong to either of the sets $A$ and $B$, and thus it precedes all the elements in both of the chains. Even the top elements of the chains may be comparable. It is also possible that one of the chains is empty (or even both of them are). In the case when $x$ covers one element or no elements at all we have $A=B=\emptyset$.
\end{remark}

The following theorem shows that the concept of an $A$-set is in fact related to sets in which double-chains are generated. For this purpose we are also going to need the concept of a $\wedge$-tree set $S$. The set $S\subseteq P$ is said to be a $\wedge$-tree set if the Hasse diagram of $\mathrm{meetcl}(S)$ is a tree, see \cite[Definition 4.1]{MatHau2}.

\begin{theorem}\label{th:A-set}
If $S$ is an $A$-set, then every element $x_i\in S$ generates a double-chain set in $S$. If $x_j$ is a maximal element in $S$ and $S\setminus\{x_j\}$ is an $A$-set, then every element $x_i\in S$ generates a double-chain set in $S$.
\end{theorem}

\begin{proof}
Since every $A$-set $S$ is a $\wedge$-tree set (see \cite[Theorem 4.3]{MatHau2}), it follows that for each $x_i\in S$ the set $C_S(x_i)$ consists of at most one element (see \cite[Lemma 4.1]{MatHau2}). This implies that the element $x_i$ generates a double-chain set in $S$.

Suppose next that $x_j$ is a maximal element in $S$ and $S\setminus\{x_j\}$ is an $A$-set. Let $x_i\in S.$ If $x_i\neq x_j$, then $x_i$ covers at most one element and $x_i$ trivially generates a double-chain set in $S$. We only need to show that the element $x_j$ generates a double-chain set in $S$ as well. Clearly
\[
\mathrm{meetcl}(C_S(x_j))\setminus C_S(x_j):=A_j\subseteq A=\{x_k\wedge x_l\,\big|\,1\leq k<l\leq n\},
\]
where $A$ is a chain by the definition of an $A$-set. Now we may choose $B_i=\emptyset$ and thus $x_j$ generates a double-chain set in $S$. 
\end{proof}

\begin{example}\label{ex1}
Let us consider a couple of meet semilattices and determine whether their elements generate double-chain sets or not. In Figure \ref{fig:ex1} (a) we have a typical $A$-set in which, as we saw in the proof of Theorem \ref{th:A-set}, any element $x_i$ trivially generates a double-chain set. In Figure \ref{fig:ex1} (b) there is a semilattice that has been obtained from an $A$-set by adding a maximum element $x_i$. By Theorem \ref{th:A-set} also in this semilattice any element $x_i$ generates a double-chain set. Figure \ref{fig:ex1} (c) shows a typical case when both of the chains $A_i$ and $B_i$ need to be nonempty. In Figure \ref{fig:ex1} (d) there is a case where $x_i$ generates a slightly more complicated double-chain set. And finally, in Figure \ref{fig:ex1} (e) we have the cube semilattice, in which the top element clearly does not generate a double-chain set. Later it will be shown (see Theorem \ref{th:B-L-conjecture}) that this is in fact the minimal example of a semilattice in which one of the elements does not generate a double-chain set. 
\end{example}

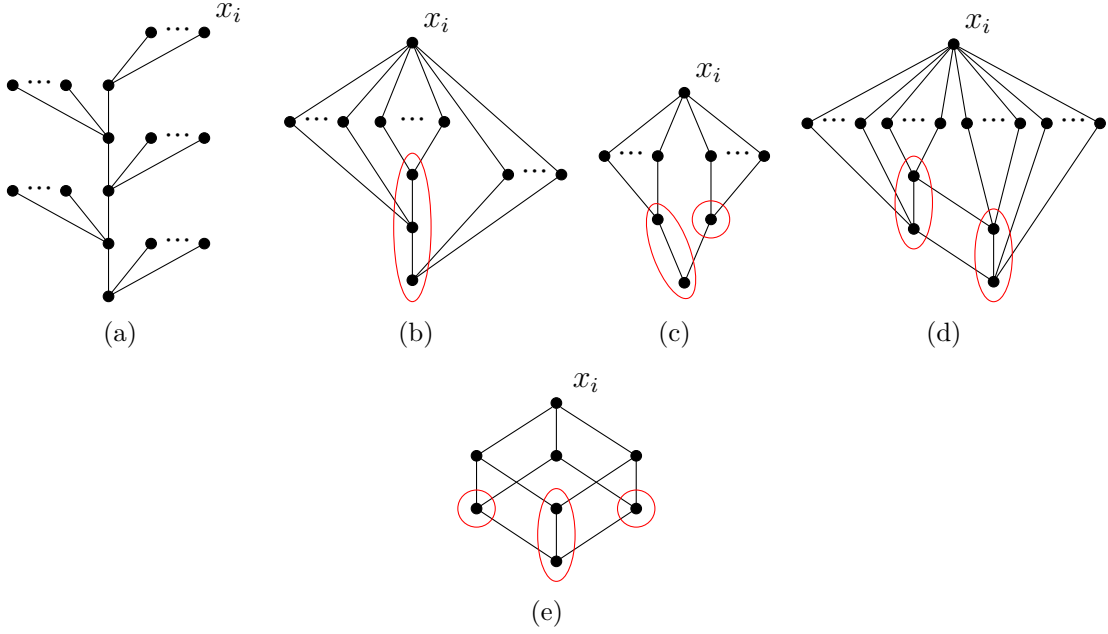
\begin{figure}[htb!]
\centering
\subfigure[\ ]
{
\begin{tikzpicture}[scale=0.7]
\draw (0,0)--(0,4);
\draw [fill] (0,0) circle [radius=0.1];
\draw [fill] (0,1) circle [radius=0.1];
\draw [fill] (0,2) circle [radius=0.1];
\draw [fill] (0,3) circle [radius=0.1];
\draw [fill] (0,4) circle [radius=0.1];
\draw [fill] (0.8,1) circle [radius=0.1];
\draw [fill] (1.8,1) circle [radius=0.1];
\draw [fill] (-0.8,2) circle [radius=0.1];
\draw [fill] (-1.8,2) circle [radius=0.1];
\draw [fill] (0.8,3) circle [radius=0.1];
\draw [fill] (1.8,3) circle [radius=0.1];
\draw [fill] (-0.8,4) circle [radius=0.1];
\draw [fill] (-1.8,4) circle [radius=0.1];
\draw [fill] (0.8,5) circle [radius=0.1];
\draw [fill] (1.8,5) circle [radius=0.1];
\draw (0,0)--(0.8,1);
\draw (0,0)--(1.8,1);
\draw (0,1)--(-0.8,2);
\draw (0,1)--(-1.8,2);
\draw (0,2)--(0.8,3);
\draw (0,2)--(1.8,3);
\draw (0,3)--(-0.8,4);
\draw (0,3)--(-1.8,4);
\draw (0,4)--(0.8,5);
\draw (0,4)--(1.8,5);
\node at (1.3,1) {$\cdots$};
\node at (-1.3,2) {$\cdots$};
\node at (1.3,3) {$\cdots$};
\node at (-1.3,4) {$\cdots$};
\node at (1.3,5) {$\cdots$};
\node [above right] at (1.8,5) {$x_i$};
\end{tikzpicture}
}
\subfigure[\ ]
{
\begin{tikzpicture}[scale=0.7]
\draw (0,2)--(0,4);
\draw [fill] (0,2) circle [radius=0.1];
\draw [fill] (0,3) circle [radius=0.1];
\draw [fill] (0,4) circle [radius=0.1];
\draw [fill] (1.8,4) circle [radius=0.1];
\draw [fill] (2.8,4) circle [radius=0.1];
\draw [fill] (-1.3,5) circle [radius=0.1];
\draw [fill] (-2.3,5) circle [radius=0.1];
\draw [fill] (-0.6,5) circle [radius=0.1];
\draw [fill] (0.6,5) circle [radius=0.1];
\draw [fill] (0,6.5) circle [radius=0.1];
\draw (0,2)--(1.8,4);
\draw (0,2)--(2.8,4);
\draw (0,3)--(-1.3,5);
\draw (0,3)--(-2.3,5);
\draw (0,4)--(0.6,5);
\draw (0,4)--(-0.6,5);
\draw (0,6.5)--(1.8,4);
\draw (0,6.5)--(2.8,4);
\draw (0,6.5)--(-1.3,5);
\draw (0,6.5)--(-2.3,5);
\draw (0,6.5)--(0.6,5);
\draw (0,6.5)--(-0.6,5);
\node at (2.3,4) {$\cdots$};
\node at (-1.8,5) {$\cdots$};
\node at (0,5) {$\cdots$};
\node [above right] at (0,6.5) {$x_i$};
\draw[red] (0,3) ellipse (10 pt and 40 pt);
\end{tikzpicture}
}
\subfigure[\ ]
{
\begin{tikzpicture}[scale=0.7]
\draw (0,2)--(-0.5,3.2)--(-0.5,4.4)--(0,5.6);
\draw (0,2)--(0.5,3.2)--(0.5,4.4)--(0,5.6);
\draw (-0.5,3.2)--(-1.5,4.4)--(0,5.6);
\draw (0.5,3.2)--(1.5,4.4)--(0,5.6);
\draw [fill] (0.5,4.4) circle [radius=0.1];
\draw [fill] (0,5.6) circle [radius=0.1];
\draw [fill] (0,2) circle [radius=0.1];
\draw [fill] (1.5,4.4) circle [radius=0.1];
\draw [fill] (-1.5,4.4) circle [radius=0.1];
\draw [fill] (-0.5,3.2) circle [radius=0.1];
\draw [fill] (0.5,3.2) circle [radius=0.1];
\draw [fill] (-0.5,4.4) circle [radius=0.1];
\node [above right] at (0,5.6) {$x_i$};
\node at (-1,4.4) {$\cdots$};
\node at (1,4.4) {$\cdots$};
\draw[rotate around={20:(-0.25,2.6)}, red] (-0.25,2.6) ellipse (10 pt and 27 pt);
\draw[red] (0.5,3.2) ellipse (10 pt and 10 pt);
\end{tikzpicture}
}
\subfigure[\ ]
{
\begin{tikzpicture}[scale=0.7]
\draw (3.5,0)--(3.5,1);
\draw (3.5,0)--(2,1);
\draw (2,1)--(2,2);
\draw (3.5,1)--(2,2);
\draw (2,1)--(0,3);
\draw (2,1)--(1,3);
\draw (2,2)--(1.5,3);
\draw (2,2)--(2.5,3);
\draw (3.5,1)--(3,3);
\draw (3.5,1)--(4,3);
\draw (3.5,0)--(4.5,3);
\draw (3.5,0)--(5.5,3);
\draw (0,3)--(2.75,4.5);
\draw (1,3)--(2.75,4.5);
\draw (1.5,3)--(2.75,4.5);
\draw (2.5,3)--(2.75,4.5);
\draw (3,3)--(2.75,4.5);
\draw (4,3)--(2.75,4.5);
\draw (4.5,3)--(2.75,4.5);
\draw (5.5,3)--(2.75,4.5);
\node [above right] at (2.75,4.5) {$x_i$};
\draw [fill] (3.5,0) circle [radius=0.1];
\draw [fill] (3.5,1) circle [radius=0.1];
\draw [fill] (2,1) circle [radius=0.1];
\draw [fill] (2,2) circle [radius=0.1];
\draw [fill] (0,3) circle [radius=0.1];
\draw [fill] (1,3) circle [radius=0.1];
\draw [fill] (1.5,3) circle [radius=0.1];
\draw [fill] (2.5,3) circle [radius=0.1];
\draw [fill] (3,3) circle [radius=0.1];
\draw [fill] (4,3) circle [radius=0.1];
\draw [fill] (4.5,3) circle [radius=0.1];
\draw [fill] (5.5,3) circle [radius=0.1];
\draw [fill] (2.75,4.5) circle [radius=0.1];
\node at (0.5,3) {$\cdots$};
\node at (2,3) {$\cdots$};
\node at (3.5,3) {$\cdots$};
\node at (5,3) {$\cdots$};
\draw[red] (2,1.5) ellipse (10 pt and 25 pt);
\draw[red] (3.5,0.5) ellipse (10 pt and 25 pt);
\end{tikzpicture}
}
\subfigure[\ ]
{
\begin{tikzpicture}[scale=0.7]
\draw (1,0)--(-0.5,1)--(-0.5,2)--(0.92,2.92);
\draw (1,0)--(1,1)--(-0.5,2);
\draw (-0.5,1)--(1,2)--(1,2.9);
\draw (1,0)--(2.5,1)--(2.5,2)--(1.08,2.92);
\draw (1,1)--(2.5,2);
\draw (2.5,1)--(1,2);
\draw [fill] (-0.5,1) circle [radius=0.1];
\draw [fill] (1,1) circle [radius=0.1];
\draw [fill] (2.5,1) circle [radius=0.1];
\draw [fill] (1,0) circle [radius=0.1];
\draw [fill] (-0.5,2) circle [radius=0.1];
\draw [fill] (1,2) circle [radius=0.1];
\draw [fill] (2.5,2) circle [radius=0.1];
\draw [fill] (1,3) circle [radius=0.1];
\node [above right] at (1.1,3) {$x_i$};
\draw[red] (-0.5,1) ellipse (10 pt and 10 pt);
\draw[red] (2.5,1) ellipse (10 pt and 10 pt);
\draw[red] (1,0.5) ellipse (10 pt and 25 pt);
\end{tikzpicture}
}
\caption{Meet semilattices related to Example \ref{ex1}. Suggestions how to choose the chains $A_i$ and $B_i$ are also marked in each figure (note that in the first case both of the chains are empty and in the last case there would have to be three chains).}\label{fig:ex1}
\end{figure}

The following theorem gives an alternative characterization for an element $x_i\in S$ to generate a double-chain set in $S$.

\begin{theorem}\label{th:width}
An element $x_i\in S$ generates a double-chain set in $S$ if and only if the width of the set $\mathrm{meetcl}(C_S(x_i))\setminus C_S(x_i)$ is less than or equal to $2$.
\end{theorem}

\begin{proof}
If the element $x_i$ generates a double-chain set in $S$, then it follows straight from the definition that the width of the set $\mathrm{meetcl}(C_S(x_i))\setminus C_S(x_i)$ is less than or equal to $2$. Suppose then that $x_i\in S$ and the width of the set $\mathrm{meetcl}(C_S(x_i))\setminus C_S(x_i)$ is less than or equal to $2$. If the width is equal to $0$, then $x_i$ covers exactly one element and we may take $A_i=B_i=\emptyset$. If the width is equal to $1$, then every two elements of $\mathrm{meetcl}(C_S(x_i))\setminus C_S(x_i)$ are comparable and thus this set is a chain. Therefore we may take $A_i=\mathrm{meetcl}(C_S(x_i))\setminus C_S(x_i)$ and $B_i=\emptyset$. For the last we assume that the width of $\mathrm{meetcl}(C_S(x_i))\setminus C_S(x_i)$ is equal to $2$. We construct the chains $A_i$ and $B_i$ from bottom to up. We begin by setting $\bigwedge C_S(x_i)\in A_i$. This element needs to be covered by exactly two elements $a_1$ and $b_1$ (otherwise it could not be the greatest lower bound for the set $C_S(x_i)$). We set $a_1\in A_i$ and $b_1\in B_i$. Next we take into consideration all the elements of $\mathrm{meetcl}(C_S(x_i))\setminus C_S(x_i)$ that cover either $a_1$ or $b_1$. There can be at most two such elements, which means that the following cases are possible:
\begin{enumerate}
\item In $\mathrm{meetcl}(C_S(x_i))\setminus C_S(x_i)$ one element covers $a_1$ and no element covers $b_1$ (Figure \ref{fig:cases} (a)). This new element is added to the chain $A_i$ and the process continues with this chain, no more elements are added to the chain $B_i$.
\item In $\mathrm{meetcl}(C_S(x_i))\setminus C_S(x_i)$ one element covers $b_1$ and no element covers $a_1$ (Figure \ref{fig:cases} (b)). This new element is added to the chain $B_i$ and the process continues with this chain, no more elements are added to the chain $A_i$.
\item In $\mathrm{meetcl}(C_S(x_i))\setminus C_S(x_i)$ one element covers $a_1$ and one covers $b_1$ (Figure \ref{fig:cases} (c)). These new elements are added to chains $A_i$ and $B_i$ respectively. The process continues with both chains.
\item In $\mathrm{meetcl}(C_S(x_i))\setminus C_S(x_i)$ one element covers either $a_1$ or $b_1$, another element covers both of them (Figure \ref{fig:cases} (d) and (e)). The first mentioned element is added to the chain $A_i$ if it covers the element $a_1$, otherwise it is added to the chain $B_i$. The other element is added to the remaining chain. The process continues with both chains.
\end{enumerate}
There are no other possibilities, since the set $\mathrm{meetcl}(C_S(x_i))\setminus C_S(x_i)$ is a meet semilattice (a meet semilattice continues to be a meet semilattice although some of the maximal elements are removed) and because it cannot contain an antichain with tree elements. Since the set $\mathrm{meetcl}(C_S(x_i))\setminus C_S(x_i)$ is also finite, repeating the above steps eventually leads to two disjoint chains $A_i$ and $B_i$, which together contain all the elements of the set $\mathrm{meetcl}(C_S(x_i))\setminus C_S(x_i)$.
\end{proof}

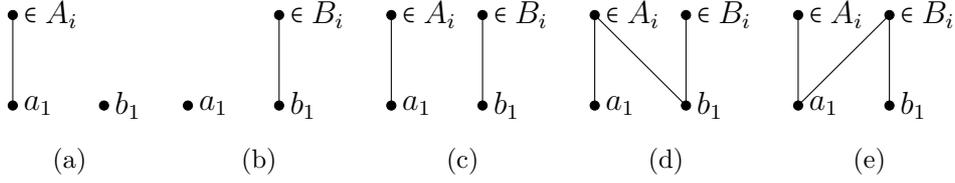
\begin{figure}[htb!]
\centering
\subfigure[\ ]
{
\begin{tikzpicture}[scale=0.6]
\draw (0,0)--(0,2);
\draw [fill] (0,0) circle [radius=0.1];
\draw [fill] (0,2) circle [radius=0.1];
\draw [fill] (2,0) circle [radius=0.1];
\node [right] at (0,0) {$a_1$};
\node [right] at (2,0) {$b_1$};
\node [right] at (0,2) {$\in A_i$};
\end{tikzpicture}
}
\subfigure[\ ]
{
\begin{tikzpicture}[scale=0.6]
\draw (2,0)--(2,2);
\draw [fill] (0,0) circle [radius=0.1];
\draw [fill] (2,2) circle [radius=0.1];
\draw [fill] (2,0) circle [radius=0.1];
\node [right] at (0,0) {$a_1$};
\node [right] at (2,0) {$b_1$};
\node [right] at (2,2) {$\in B_i$};
\end{tikzpicture}
}
\subfigure[\ ]
{
\begin{tikzpicture}[scale=0.6]
\draw (2,0)--(2,2);
\draw (0,0)--(0,2);
\draw [fill] (0,0) circle [radius=0.1];
\draw [fill] (0,2) circle [radius=0.1];
\draw [fill] (2,2) circle [radius=0.1];
\draw [fill] (2,0) circle [radius=0.1];
\node [right] at (0,0) {$a_1$};
\node [right] at (2,0) {$b_1$};
\node [right] at (0,2) {$\in A_i$};
\node [right] at (2,2) {$\in B_i$};
\end{tikzpicture}
}
\subfigure[\ ]
{
\begin{tikzpicture}[scale=0.6]
\draw (2,0)--(2,2);
\draw (0,0)--(0,2);
\draw (0,2)--(2,0);
\draw [fill] (0,0) circle [radius=0.1];
\draw [fill] (0,2) circle [radius=0.1];
\draw [fill] (2,2) circle [radius=0.1];
\draw [fill] (2,0) circle [radius=0.1];
\node [right] at (0,0) {$a_1$};
\node [right] at (2,0) {$b_1$};
\node [right] at (0,2) {$\in A_i$};
\node [right] at (2,2) {$\in B_i$};
\end{tikzpicture}
}
\subfigure[\ ]
{
\begin{tikzpicture}[scale=0.6]
\draw (2,0)--(2,2);
\draw (0,0)--(0,2);
\draw (2,2)--(0,0);
\draw [fill] (0,0) circle [radius=0.1];
\draw [fill] (0,2) circle [radius=0.1];
\draw [fill] (2,2) circle [radius=0.1];
\draw [fill] (2,0) circle [radius=0.1];
\node [right] at (0,0) {$a_1$};
\node [right] at (2,0) {$b_1$};
\node [right] at (0,2) {$\in A_i$};
\node [right] at (2,2) {$\in B_i$};
\end{tikzpicture}
}
\caption{Illustrations of different cases in the proof of Theorem \ref{th:width}.}\label{fig:cases}
\end{figure}

Next we need to develop some further terminology. Suppose that $x_i$ generates a double-chain set in $S$, and that $x_i$ covers at least two elements in $S$. Now let $A_i$ and $B_i$ be two disjoint chains such that $\mathrm{meetcl}(C_S(x_i))\setminus C_S(x_i)=A_i\cup B_i$. For each element $z\in C_S(x_i)$ there exists an element $a\in A_i$ or $b\in B_i$ such that $a\lessdot z$ or $b\lessdot z$ in $\mathrm{meetcl}(C_S(x_i))$. In the first case we say that $z$ \emph{attaches to the chain $A_i$} and in the second case it attaches to the chain $B_i$. Sometimes it is even possible that some element is attached to both chains. However, the next lemma shows that there can be only one such element.

\begin{lemma}\label{max1}
If $x_i$ generates a double-chain set in $S$, where $A_i$ and $B_i$ are the corresponding two chains, then there can be at most one element $z\in C_S(x_i)$ that attaches to both chains $A_i$ and $B_i$.
\end{lemma} 

\begin{proof}
Suppose for a contradiction that there exist two distinct elements $z,z'\in C_S(x_i)$ such that $a\lessdot z$, $b\lessdot z$, $a'\lessdot z'$ and $b'\lessdot z'$, where $a,a'\in A_i$ and $b,b'\in B_i$. This means that $\min(a,a')$ and $\min(b,b')$ are both common lower bounds for $z$ and $z'$. No greater common lower bound can be found from either chain, since otherwise either $z$ or $z'$ would cover the element $z\wedge z'$ instead of $\min(a,a')$ or $\min(b,b')$. Therefore either of the elements $\min(a,a')$ or $\min(b,b')$ needs to be equal to $z\wedge z'$, and thus the elements $\min(a,a')$ and $\min(b,b')$ must be comparable. The following possibilities come into question:
\begin{enumerate}
\item If $a=\min(a,a')$ and $b=\min(b,b')$ (see Figure \ref{fig:cases2} (a)), then either $a\prec b$ or $b\prec a$. This is impossible, since $z$ was supposed to cover both of these elements.
\item If $a'=\min(a,a')$ and $b'=\min(b,b')$ (see Figure \ref{fig:cases2} (b)), then either $a'\prec b'$ or $b'\prec a'$. This is impossible, since $z'$ was supposed to cover both of these elements.
\item If $a=\min(a,a')$ and $b'=\min(b,b')$ (see Figure \ref{fig:cases2} (c)), then either $a\prec b'$ or $b'\prec a$. In the first case $z$ covers comparable elements $a$ and $b$, in the second case $z'$ covers comparable elements $a'$ and $b'$. In both cases we have a contradiction.
\item If $a'=\min(a,a')$ and $b=\min(b,b')$ (see Figure \ref{fig:cases2} (d)), then either $a'\prec b$ or $b\prec a'$. As in the part 3 we have a contradiction, since either $z$ or $z'$ now covers two comparable elements.
\end{enumerate}
Each of the cases 1--4 yields a contradiction, and thus we have proven the claim.
\end{proof}

\begin{figure}[htb!]
\centering
\subfigure[\ ]
{
\begin{tikzpicture}[scale=0.9]
\draw (0.5,0)--(1,2);
\draw (0.5,1)--(2,2);
\draw (2.5,1)--(2,2);
\draw (2.5,0)--(1,2);
\draw (2.5,0)--(2.5,1);
\draw (0.5,0)--(0.5,1);
\draw[dotted] (0.5,0)--(2.5,0);
\draw [fill] (0.5,0) circle [radius=0.1];
\draw [fill] (0.5,1) circle [radius=0.1];
\draw [fill] (1,2) circle [radius=0.1];
\draw [fill] (2,2) circle [radius=0.1];
\draw [fill] (2.5,1) circle [radius=0.1];
\draw [fill] (2.5,0) circle [radius=0.1];
\node [right] at (1,2) {$z$};
\node [right] at (2,2) {$z'$};
\node [left] at (0.5,0) {$a$};
\node [left] at (0.5,1) {$a'$};
\node [right] at (2.5,0) {$b$};
\node [right] at (2.5,1) {$b'$};
\end{tikzpicture}
}
\subfigure[\ ]
{
\begin{tikzpicture}[scale=0.9]
\draw (0.5,0)--(2,2);
\draw (0.5,1)--(1,2);
\draw (2.5,1)--(1,2);
\draw (2.5,0)--(2,2);
\draw (2.5,0)--(2.5,1);
\draw (0.5,0)--(0.5,1);
\draw[dotted] (0.5,0)--(2.5,0);
\draw [fill] (0.5,0) circle [radius=0.1];
\draw [fill] (0.5,1) circle [radius=0.1];
\draw [fill] (1,2) circle [radius=0.1];
\draw [fill] (2,2) circle [radius=0.1];
\draw [fill] (2.5,1) circle [radius=0.1];
\draw [fill] (2.5,0) circle [radius=0.1];
\node [right] at (1,2) {$z$};
\node [right] at (2,2) {$z'$};
\node [left] at (0.5,0) {$a'$};
\node [left] at (0.5,1) {$a$};
\node [right] at (2.5,0) {$b'$};
\node [right] at (2.5,1) {$b$};
\end{tikzpicture}
}
\subfigure[\ ]
{
\begin{tikzpicture}[scale=0.9]
\draw (0.5,0)--(1,2);
\draw (0.5,1)--(2,2);
\draw (2.5,1)--(1,2);
\draw (2.5,0)--(2,2);
\draw (2.5,0)--(2.5,1);
\draw (0.5,0)--(0.5,1);
\draw[dotted] (0.5,0)--(2.5,0);
\draw [fill] (0.5,0) circle [radius=0.1];
\draw [fill] (0.5,1) circle [radius=0.1];
\draw [fill] (1,2) circle [radius=0.1];
\draw [fill] (2,2) circle [radius=0.1];
\draw [fill] (2.5,1) circle [radius=0.1];
\draw [fill] (2.5,0) circle [radius=0.1];
\node [right] at (1,2) {$z$};
\node [right] at (2,2) {$z'$};
\node [left] at (0.5,0) {$a$};
\node [left] at (0.5,1) {$a'$};
\node [right] at (2.5,0) {$b'$};
\node [right] at (2.5,1) {$b$};
\end{tikzpicture}
}
\subfigure[\ ]
{
\begin{tikzpicture}[scale=0.9]
\draw (0.5,0)--(2,2);
\draw (0.5,1)--(1,2);
\draw (2.5,1)--(2,2);
\draw (2.5,0)--(1,2);
\draw (2.5,0)--(2.5,1);
\draw (0.5,0)--(0.5,1);
\draw[dotted] (0.5,0)--(2.5,0);
\draw [fill] (0.5,0) circle [radius=0.1];
\draw [fill] (0.5,1) circle [radius=0.1];
\draw [fill] (1,2) circle [radius=0.1];
\draw [fill] (2,2) circle [radius=0.1];
\draw [fill] (2.5,1) circle [radius=0.1];
\draw [fill] (2.5,0) circle [radius=0.1];
\node [right] at (1,2) {$z$};
\node [right] at (2,2) {$z'$};
\node [left] at (0.5,0) {$a'$};
\node [left] at (0.5,1) {$a$};
\node [right] at (2.5,0) {$b$};
\node [right] at (2.5,1) {$b'$};
\end{tikzpicture}
}
\caption{Illustrations of different cases in the proof of Theorem \ref{max1}.}\label{fig:cases2}
\end{figure}
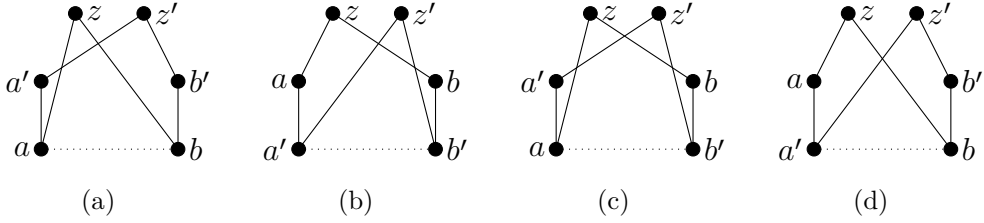

Next we take a closer look at the case when there exists an element $z\in C_S(x_i)$ attaching to both of the chains. The existence of such element has certain implications on the elements of $A_i$ and $B_i$, more precisely on their ability to precede one another. The following lemma explains this. 

\begin{lemma}\label{tasan1}
Suppose that $x_i$ generates a double-chain set in $S$, where $A_i$ and $B_i$ are the corresponding two chains, and that there exists (exactly one) element $z\in C_S(x_i)$ attacheing to both of the chains $A_i$ and $B_i$. Moreover, let $a\in A_i$ and $b\in B_i$ be the elements satisfying $a\lessdot x_i$ and $b\lessdot x_i$. If $a'\prec b'$ or $b'\prec a'$ for some $a'\in A_i$ and $b'\in B_i$, then we must have $a'\preceq a$ and $b'\preceq b$, where at least one of these relations is strict.
\end{lemma}

\begin{proof}
We may assume that $a'\lessdot b'$ in $\mathrm{meetcl}(C_S(x_i))$ for some $a'\in A_i$ and $b'\in B_i$. Let $y\in C_S(x_i)$ be any element attaching to $x_b$, the top element of $B_i$ (clearly $b\preceq x_b$). We may assume that $y\neq z$ since there have to be at least two elements in $C_S(x_i)$ that are attached to $x_b$.

Let us show that $a'\preceq a$. If $a\prec a'$, then $a$ and $b$ are both common lower bounds for $y$ and $z$, and because $a\lessdot z$ and $b\lessdot z$, no greater common lower bound can be found from either chain. Since $y\wedge z$ is in either of the chains, $y\wedge z$ is equal to $a$ or $b$ and thus the elements $a$ and $b$ must be comparable (see Figure \ref{fig:cases3} (a) and (b)). This would be a contradiction, since $z$ cannot be attached to two comparable elements. We may thus deduce that $a'\preceq a$.

Next we are going to show that $b'\preceq b$. Suppose for a contradiction that $b\prec b'$ (the situation is illustrated in Figure \ref{fig:cases3} (c)). In this case $a'$ and $b$ are common lower bounds for $z$ and $y$. No greater common lower bound can be found from the chain $B_i$, since otherwise $b\lessdot z$ would not hold. On the other hand, if we had $y\wedge z\in A_i$ with $a'\prec y\wedge z$, then it would imply that $b\prec z\wedge y\preceq a$. This would lead to a contradiction, since $z$ would now cover two comparable elements $a$ and $b$. We may deduce that either of the elements $a'$ and $b$ is equal to $y\wedge z$ and thus we must have $b\prec a'$ or $a'\prec b$. The first case is impossible since, again, $z$ cannot cover two comparable elements $a$ and $b$. Also the other case yields a contradiction since in this case we have $a'\prec b\prec b'$, which means that $a'\not\lessdot b'$. The last part of the claim follows immediately, since the elements $a$ and $b$ attached to $z$ need to be incomparable. 
\end{proof}

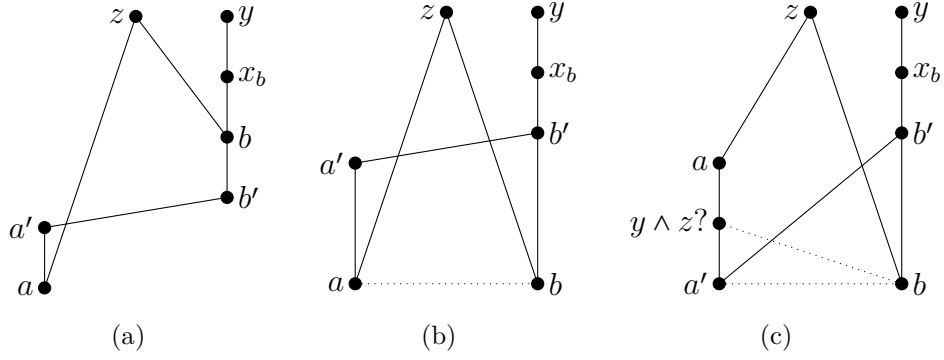
\begin{figure}[htb!]
\centering
\subfigure[\ ]
{
\begin{tikzpicture}[scale=0.8]
\draw (0,0.5)--(0,1.5);
\draw (3,2)--(0,1.5);
\draw (3,2)--(3,4);
\draw (3,4)--(3,5);
\draw (3,3)--(1.5,5);
\draw (0,0.5)--(1.5,5);
\draw [fill] (0,0.5) circle [radius=0.1];
\draw [fill] (0,1.5) circle [radius=0.1];
\draw [fill] (3,2) circle [radius=0.1];
\draw [fill] (3,3) circle [radius=0.1];
\draw [fill] (3,4) circle [radius=0.1];
\draw [fill] (3,5) circle [radius=0.1];
\draw [fill] (1.5,5) circle [radius=0.1];
\node [left] at (0,0.5) {$a$};
\node [left] at (0,1.5) {$a'$};
\node [right] at (3,2) {$b'$};
\node [right] at (3,3) {$b$};
\node [right] at (3,4) {$x_b$};
\node [right] at (3,5) {$y$};
\node [left] at (1.5,5) {$z$};
\end{tikzpicture}
}
\subfigure[\ ]
{
\begin{tikzpicture}[scale=0.8]
\draw (0,0.5)--(0,2.5);
\draw (3,3)--(0,2.5);
\draw (3,0.5)--(3,4);
\draw (3,4)--(3,5);
\draw (3,0.5)--(1.5,5);
\draw (0,0.5)--(1.5,5);
\draw[dotted] (0,0.5)--(3,0.5);
\draw [fill] (0,0.5) circle [radius=0.1];
\draw [fill] (0,2.5) circle [radius=0.1];
\draw [fill] (3,0.5) circle [radius=0.1];
\draw [fill] (3,3) circle [radius=0.1];
\draw [fill] (3,4) circle [radius=0.1];
\draw [fill] (3,5) circle [radius=0.1];
\draw [fill] (1.5,5) circle [radius=0.1];
\node [left] at (0,0.5) {$a$};
\node [left] at (0,2.5) {$a'$};
\node [right] at (3,0.5) {$b$};
\node [right] at (3,3) {$b'$};
\node [right] at (3,4) {$x_b$};
\node [right] at (3,5) {$y$};
\node [left] at (1.5,5) {$z$};
\end{tikzpicture}
}
\subfigure[\ ]
{
\begin{tikzpicture}[scale=0.8]
\draw (0,0.5)--(0,2.5);
\draw (3,3)--(0,0.5);
\draw (3,0.5)--(3,4);
\draw (3,4)--(3,5);
\draw (3,0.5)--(1.5,5);
\draw (0,2.5)--(1.5,5);
\draw[dotted] (0,0.5)--(3,0.5);
\draw[dotted] (3,0.5)--(0,1.5);
\draw [fill] (0,0.5) circle [radius=0.1];
\draw [fill] (0,2.5) circle [radius=0.1];
\draw [fill] (3,0.5) circle [radius=0.1];
\draw [fill] (3,3) circle [radius=0.1];
\draw [fill] (3,4) circle [radius=0.1];
\draw [fill] (3,5) circle [radius=0.1];
\draw [fill] (1.5,5) circle [radius=0.1];
\draw [fill] (0,1.5) circle [radius=0.1];
\node [left] at (0,0.5) {$a'$};
\node [left] at (0,2.5) {$a$};
\node [right] at (3,0.5) {$b$};
\node [right] at (3,3) {$b'$};
\node [right] at (3,4) {$x_b$};
\node [right] at (3,5) {$y$};
\node [left] at (1.5,5) {$z$};
\node [left] at (0,1.5) {$y\wedge z$?};
\end{tikzpicture}
}
\caption{Illustrations of different cases in the proof of Theorem \ref{tasan1}.}\label{fig:cases3}
\end{figure}

\section{Poset-theoretic Möbius function on meet closed double-chain sets}\label{sect:Möbius}

From now on we assume that the set $S\subseteq P$ is meet closed (in other words, $(S,\preceq)$ is a meet semilattice). The poset-theoretic Möbius function of the structure $(S,\preceq)$ is often defined as the inverse of the incidence function $\zeta$ with respect to the convolution operation. However, if we wish to calculate the value $\mu_S(x_j,x_i)$, the easiest way would probably be to use the following recursive formula, which follows directly from the definition:
\begin{align*}
&\mu_S(x_i,x_i)=1,\\
&\mu_S(x_j,x_i)=-\sum_{x_j\prec x_k\preceq x_i}\mu_S(x_k,x_i)=-\sum_{x_j\preceq x_k\prec x_i}\mu_S(x_j,x_k).
\end{align*}
If $x_i$ generates a double-chain set in, then calculating values $\mu_S(x_j,x_i)$ for different elements $x_j$ becomes rather simple. The following theorem shows this.

\begin{theorem}\label{th:mobius}
Suppose that $S$ is meet closed and $x_i$ generates a double-chain set in $S$. Denote $\mathrm{meetcl}(C_S(x_i))\setminus C_S(x_i)=A_i\cup B_i\subset S$, where $A_i$ and $B_i$ are chains with $A_i\cap B_i=\emptyset$. Let $x_a$ and $x_b$ denote the top elements of $A_i$ and $B_i$, respectively (in the case when one of the chains is empty we agree that $x_a=x_b$). For each $x_k\in A_i$ or $x_k\in B_i$, let $\eta(x_k)$ denote the number of elements $z\in C_S(x_i)$ attached to $x_k$. If no element is attached to both of the chains, then

\begin{numcases}{\mu_S(x_j,x_i)=}
1 & if $x_j=x_i$,\label{1a}\\
-1 & if  $x_j\in C_S(x_i)$,\label{1b}\\
\eta(x_j)-1 & if $x_j\in\{x_a,x_b\}$ is maximal in $A_i\cup B_i$,\label{1c}\\ \nonumber
\eta(x_j) & if $x_j\in \{x_a,x_b\}$ is not maximal\\
\ & \qquad in $A_i\cup B_i$\label{1d}\\\nonumber
\eta(x_j) & if $x_j\in A_i\cup B_i$ is not maximal\\ 
\ & \qquad in $A_i\cup B_i$ and $x_j\neq x_a\wedge x_b$,\label{1e}\\\nonumber
\eta(x_j)+1 & if $x_a$ and $x_b$ are incomparable\\
\ & \qquad and $x_j= x_a\wedge x_b\in A_i\cup B_i$,\label{1f}\\
0 & otherwise\label{1g}.
\end{numcases}

If (exactly) one element $x_p\in C_S(x_i)$ is attached to some element $x_q\in A_i$ and to some element $x_r\in B_i$, then

\begin{numcases}{\mu_S(x_j,x_i)=}
1 & if $x_j=x_i$,\label{2a}\\
-1 & if $x_j\in C_S(x_i)$,\label{2b}\\
\eta(x_j)-1 & if $x_j=x_a$ or $x_j=x_b$,\label{2c}\\
\eta(x_j) & if $x_j\in A_i\cup B_i$, $x_j\neq x_a$ and $x_j\neq x_b$,\label{2d}\\
0 & otherwise\label{2e}.
\end{numcases}

\end{theorem}

\begin{proof}
First we take care of Equations \eqref{1g} and \eqref{2e} by noting that $\mu_S(x_j,x_i)$ may be nonzero only if $x_j=x_i$ or $x_j\in\mathrm{meetcl}(C_S(x_i))$ (see \cite[Lemma 3.2]{MatHau3}).  The recursive formula for the Möbius function implies that $\mu_S(x_i,x_i)=1$ and $\mu_S(x_j,x_i)=-1$ for $x_j\in C_S(x_i).$ In other words, Equations \eqref{1a}, \eqref{1b}, \eqref{2a} and \eqref{2b} are also true. If $x_j$ is maximal in $A_i\cup B_i$, then we have
\[
\mu_S(x_j,x_i)=-\sum_{x_j\prec x_k\preceq x_i}\mu_S(x_k,x_i)=-(\underbrace{(-1)+\cdots+(-1)}_{\eta(x_j)\mathrm{\ times}}+\underbrace{1}_{=\mu_S(x_i,x_i)})=\eta(x_j)-1,
\]
and therefore Equations \eqref{1c} and \eqref{2c} hold as well.

Next we take all the remaining, namely the non-maximal elements $x_j$ of $A_i\cup B_i$ into consideration. Assume first that no element $x_k\in C_S(x_i)$ is attached to both of the chains $A_i$ and $B_i$. We proceed inductively and suppose that the claim holds for each element $x_k\in A_i\cup B_i$ such that $j<k\leq n$. Assume next that the top elements are comparable, say $x_a\prec x_b$, and that $x_j=x_a$. Then we have
\begin{align*}
&\mu_S(x_a,x_i)=-\sum_{x_a\prec x_k\preceq x_i}\mu_S(x_k,x_i)=-\left(\underbrace{\sum_{x_b\preceq x_k\preceq x_i}\mu_S(x_k,x_i)}_{=0}+\sum_{\substack{x_a\prec x_k\preceq x_i\\x_b\not\preceq x_k}}\mu_S(x_k,x_i)\right)\\
&=-\left(\underbrace{\eta(x_a)}_{\substack{\mathrm{this\ many\ ele-}\\\mathrm{ments\ of\ }C_S(x_i)\\\mathrm{cover\ }x_a}}\cdot(-1)+\sum_{\substack{x_k\in A_i\cup B_i\\x_a\prec x_k\\x_k\neq x_b}}\left(\underbrace{\mu_S(x_k,x_i)}_{=\eta(x_k)}+\underbrace{\eta(x_k)}_{\substack{\mathrm{this\ many\ ele-}\\\mathrm{ments\ of\ }C_S(x_i)\\\mathrm{cover\ }x_k}}\cdot(-1)\right)\right)=\eta(x_a).
\end{align*}
This takes care of \eqref{1d}.

We are now in a position to prove one part of Equation \eqref{1e}. Suppose then that exactly one of the conditions $x_j\prec x_a$ and $x_j\prec x_b$ holds. We may assume that $x_j\prec x_a$ (the other case is similar). From the recursive formula we now obtain
\begin{align*}
&\mu_S(x_j,x_i)=-\sum_{x_j\prec x_k\preceq x_i}\mu_S(x_k,x_i)=-\left(\underbrace{\sum_{x_a\preceq x_k\preceq x_i}\mu_S(x_k,x_i)}_{=0}+\sum_{\substack{x_j\prec x_k\prec x_i\\x_a\not\preceq x_k}}\mu_S(x_k,x_i)\right)\\
&=-\left(\underbrace{\eta(x_j)}_{\substack{\mathrm{this\ many\ ele-}\\\mathrm{ments\ of\ }C_S(x_i)\\\mathrm{cover\ }x_j}}\cdot(-1)+\sum_{\substack{x_k\in A_i\\x_j\prec x_k\\x_k\neq x_a}}\left(\underbrace{\mu_S(x_k,x_i)}_{=\eta(x_k)}+\underbrace{\eta(x_k)}_{\substack{\mathrm{this\ many\ ele-}\\\mathrm{ments\ of\ }C_S(x_i)\\\mathrm{cover\ }x_k}}\cdot(-1)\right)\right)=\eta(x_j).
\end{align*}
This proves the first part of \eqref{1e}.

At this point we need to deal with Equation \eqref{1f}. This situation is illustrated in Figure \ref{fig:illustration}. If $x_j=x_a\wedge x_b$ and $x_a$ and $x_b$ are incomparable, then we have
\begin{align*}
&\mu_S(x_j,x_i)=-\left(\underbrace{\sum_{x_a\preceq x_k\preceq x_i}\mu_S(x_k,x_i)}_{=0}+\underbrace{\sum_{x_b\preceq x_k\prec x_i}\mu_S(x_k,x_i)}_{=-1}+\sum_{\substack{x_j\prec x_k\prec x_i\\x_a,x_b\not\preceq x_k}}\mu_S(x_k,x_i)\right)\\
&=-\left(-1+\underbrace{\eta(x_j)}_{\substack{\mathrm{this\ many\ ele-}\\\mathrm{ments\ of\ }C_S(x_i)\\\mathrm{cover\ }x_j}}\cdot(-1)+\sum_{\substack{x_k\in A_i\cup B_i\\x_j\prec x_k\prec x_a\\x_k\prec x_b}}\left(\underbrace{\mu_S(x_k,x_i)}_{=\eta(x_k)}+\underbrace{\eta(x_k)}_{\substack{\mathrm{this\ many\ ele-}\\\mathrm{ments\ of\ }C_S(x_i)\\\mathrm{cover\ }x_k}}\cdot(-1)\right)\right)\\&=\eta(x_j)+1.
\end{align*}
Thus we have proven Equation \eqref{1f}.

The last remaining case is the one in which $x_j\neq x_a\wedge x_b$ precedes both of the elements $x_a$ and $x_b$, which means that $x_j\prec x_a\wedge x_b$. Thus we go back and prove the rest of Equation \eqref{1e}. We obtain
\begin{align*}
&\mu_S(x_j,x_i)=-\left(\underbrace{\sum_{x_a\wedge x_b\preceq x_k\preceq x_i}\mu_S(x_k,x_i)}_{=0}+\sum_{\substack{x_j\prec x_k\prec x_i\\x_k\not\preceq x_a\wedge x_b }}\mu_S(x_k,x_i)\right)\\
&=-\left(\underbrace{\eta(x_j)}_{\substack{\mathrm{this\ many\ ele-}\\\mathrm{ments\ of\ }C_S(x_i)\\\mathrm{cover\ }x_j}}\cdot(-1)+\sum_{\substack{x_k\in A_i\cup B_i\\x_j\prec x_k\prec x_i\\x_a\wedge x_b\not\preceq x_k}}\left(\underbrace{\mu_S(x_k,x_i)}_{=\eta(x_k)}+\underbrace{\eta(x_k)}_{\substack{\mathrm{this\ many\ ele-}\\\mathrm{ments\ of\ }C_S(x_i)\\\mathrm{cover\ }x_k}}\cdot(-1)\right)\right)=\eta(x_j),
\end{align*}
which completes the proof of \eqref{1e}.

Assume second that there is exactly one element $x_p\in C_S(x_i)$ such that $A_i\ni x_q\lessdot x_p$ and $B_i\ni x_r\lessdot x_p$. It now follows from Lemma \ref{tasan1} that the top elements $x_a$ and $x_b$ of $A_i$ and $B_i$ must be incomparable (clearly $x_q\preceq x_a$ and $x_r\preceq x_b$). In fact, Lemma \ref{tasan1} implies that $x_a\wedge x_b=x_q\wedge x_r$. Since the other cases in Equation \eqref{2d} can be dealt inductively as earlier, we only need to check what happens when $x_j=x_a\wedge x_b$ (see Figure \ref{fig:illustration2}). Again we assume that the claim holds for all $x_k$ with $k>j$. Now we have
\begin{align*}
&\mu_S(x_j,x_i)=-\left(\underbrace{\sum_{x_a\preceq x_k\preceq x_i}\mu_S(x_k,x_i)}_{=0}+\underbrace{\sum_{x_b\preceq x_k\prec x_i}\mu_S(x_k,x_i)}_{=-1}+\sum_{\substack{x_j\prec x_k\prec x_i\\x_a\not\preceq x_k\\x_b\not\preceq x_k}}\mu_S(x_k,x_i)\right)\\
&=-\left(-1+\underbrace{\eta(x_j)}_{\substack{\mathrm{this\ many\ ele-}\\\mathrm{ments\ of\ }C_S(x_i)\\\mathrm{cover\ }x_j}}\cdot(-1)+\overbrace{
\sum_{\substack{x_k\in A_i\cup B_i\\x_j\prec x_k\prec x_a\\x_k\prec x_b}}\Big(\underbrace{\eta(x_k)}_{=\mu_S(x_k,x_i)}+\underbrace{\eta(x_k)}_{\substack{\mathrm{this\ many\ ele-}\\\mathrm{ments\ of\ }C_S(x_i)\\\mathrm{cover\ }x_k}}\cdot(-1)
\Big)}^{\text{here }\mu_S(x_p,x_i)=-1\text{ is counted twice}}\right)\\&=-(-1-\eta(x_j)+1)=\eta(x_j).
\end{align*}
Thus \eqref{2d} holds and our proof is complete.
\end{proof}

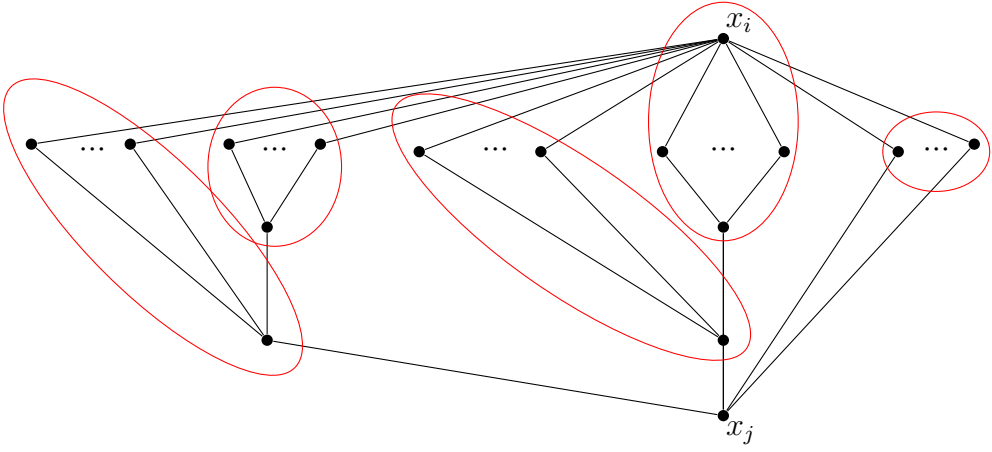
\begin{figure}[ht]
\tikzstyle{every node}=[fill,circle,inner sep=1.5pt]
\centering
\begin{tikzpicture}
\node (a) at (1.9,4.1)  {};
  \node[draw=none, fill=none] (b) at (2.7,4) {$\cdots$};
  \node (c) at (3.2,4.1)  {};
  \node (d) at (4.5,4.1)  {};
  \node[draw=none, fill=none] (e) at (5.1,4) {$\cdots$};
  \node (f) at (5.7,4.1) {};
  \node (h) at (7,4) {};
  \node[draw=none, fill=none] (i) at (8,4) {$\cdots$};
  \node (j) at (8.6,4) {};
  \node (k) at (10.2,4) {};
  \node[draw=none, fill=none] (l) at (11,4) {$\cdots$};
  \node (m) at (11.8,4) {};
  \node (n) at (13.3,4) {};
  \node[draw=none, fill=none] (o) at (13.8,4) {$\cdots$};
  \node (p) at (14.3,4.1) {};
  \node (q) at (11,3) {};
  \node (r) at (11,1.5) {};
  \node (s) at (11,0.5) {};
  \node (t) at (5,3) {};
	\node (u) at (11,5.5) {};
	\node (v) at (5,1.5) {};
	\node[draw=none, fill=none] [below right] at (s) {$x_j$};
	\node[draw=none, fill=none] [above right] at (u) {$x_i$};
	\path (a) edge (u);
	\path (c) edge (u);
	\path (d) edge (u);
	\path (f) edge (u);
	\path (h) edge (u);
	\path (j) edge (u);
	\path (k) edge (u);
	\path (m) edge (u);
	\path (n) edge (u);
	\path (p) edge (u);
	\path (s) edge (q);
	\path (s) edge (v);
	\path (a) edge (v);
	\path (c) edge (v);
	\path (s) edge (q);
	\path (t) edge (d);
	\path (f) edge (t);
	\path (k) edge (q);
	\path (m) edge (q);
	\path (t) edge (v);
	\path (h) edge (r);
	\path (j) edge (r);
	\path (s) edge (n);
	\path (s) edge (p);
	\draw[red] (11,4.4) ellipse (28 pt and 45 pt);
	\draw[red] (13.8,4) ellipse (20 pt and 15 pt);
	\draw[rotate around={135:(3.5,3)}, red] (3.5,3) ellipse (75 pt and 25 pt);
	\draw[red] (5.1,3.8) ellipse (25 pt and 30 pt);
	\draw[rotate around={145:(9,3)}, red] (9,3) ellipse (80 pt and 25 pt);
\end{tikzpicture}
\caption{Calculating the value of $\mu_S(x_j,x_i)$ when $x_j=x_a\wedge x_b$ and no element in $C_S(x_i)$ is attached to both chains. Counting together the sum of Möbius function values of each rounded collection of elements yields three zeros, one $-1$ and $-\eta(x_j)$.}
\label{fig:illustration}
\end{figure}

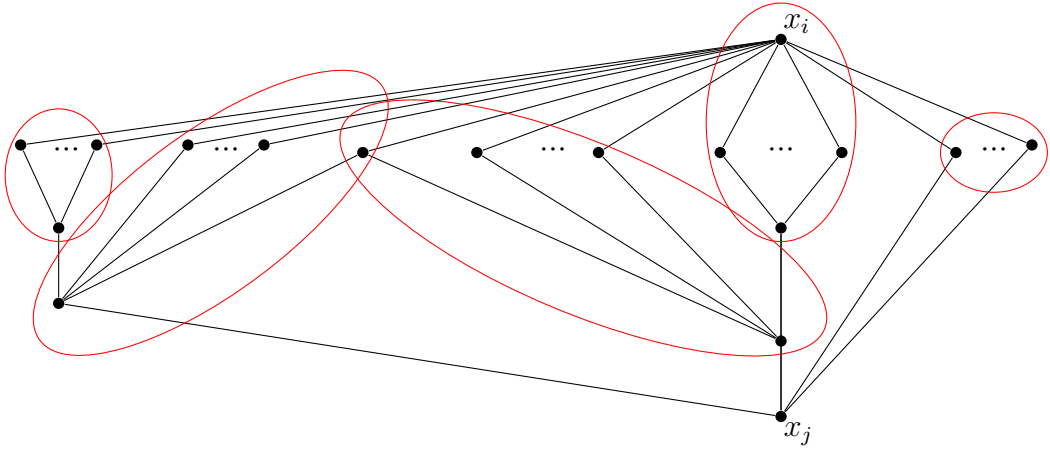
\begin{figure}[h!]
\tikzstyle{every node}=[fill,circle,inner sep=1.5pt]
\centering
\begin{tikzpicture}
\node (a) at (3.2,4.1)  {};
  \node[draw=none, fill=none] (b) at (3.7,4) {$\cdots$};
  \node (c) at (4.2,4.1)  {};
  \node (d) at (1,4.1)  {};
  \node[draw=none, fill=none] (e) at (1.6,4) {$\cdots$};
  \node (f) at (2,4.1) {};
  \node (g) at (5.5,4) {};
  \node (h) at (7,4) {};
  \node[draw=none, fill=none] (i) at (8,4) {$\cdots$};
  \node (j) at (8.6,4) {};
  \node (k) at (10.2,4) {};
  \node[draw=none, fill=none] (l) at (11,4) {$\cdots$};
  \node (m) at (11.8,4) {};
  \node (n) at (13.3,4) {};
  \node[draw=none, fill=none] (o) at (13.8,4) {$\cdots$};
  \node (p) at (14.3,4.1) {};
  \node (q) at (11,3) {};
  \node (r) at (11,1.5) {};
  \node (s) at (11,0.5) {};
  \node (t) at (1.5,3) {};
	\node (u) at (11,5.5) {};
	\node (v) at (1.5,2) {};
	\node[draw=none, fill=none] [below right] at (s) {$x_j$};
	\node[draw=none, fill=none] [above right] at (u) {$x_i$};
	\path (a) edge (u);
	\path (c) edge (u);
	\path (d) edge (u);
	\path (f) edge (u);
	\path (g) edge (u);
	\path (h) edge (u);
	\path (j) edge (u);
	\path (k) edge (u);
	\path (m) edge (u);
	\path (n) edge (u);
	\path (p) edge (u);
	\path (s) edge (q);
	\path (s) edge (v);
	\path (a) edge (v);
	\path (c) edge (v);
	\path (s) edge (q);
	\path (t) edge (d);
	\path (f) edge (t);
	\path (k) edge (q);
	\path (m) edge (q);
	\path (t) edge (v);
	\path (g) edge (v);
	\path (g) edge (r);
	\path (h) edge (r);
	\path (j) edge (r);
	\path (s) edge (n);
	\path (s) edge (p);
	\draw[red] (11,4.4) ellipse (28 pt and 45 pt);
	\draw[red] (13.8,4) ellipse (20 pt and 15 pt);
	\draw[rotate around={157:(8.4,3)}, red] (8.4,3) ellipse (98 pt and 32 pt);
	\draw[red] (1.5,3.7) ellipse (20 pt and 25 pt);
	\draw[rotate around={37:(3.5,3.2)}, red] (3.5,3.2) ellipse (80 pt and 30 pt);
\end{tikzpicture}
\caption{Calculating the value of $\mu_S(x_j,x_i)$ when $x_j=x_a\wedge x_b$ and one element in $C_S(x_i)$ is attached to both chains. Counting together the sum of Möbius function values of each collection of rounded elements yields three zeros, one $-1$ and $-\eta(x_j)$, but one $-1$ is counted twice.}
\label{fig:illustration2}
\end{figure}

\section{Invertibility of LCM matrices on GCD closed double-chain sets}\label{sect:inv}

From now on we consider the divisor lattice $(\mathbb{Z}_+,|)$. Suppose that the set $S=\{x_1,x_2,\ldots,x_n\}$ is GCD closed with $x_1<x_2<\cdots<x_n.$ In order to analyze the determinant of the LCM matrix $[S]$ it is useful to observe the following relation between the matrix $[S]$ and the reciprocal GCD matrix with $\frac{1}{\gcd(x_i,x_j)}$ as its $ij$ entry (similar observations are also made, for example, in \cite{MatHau3} and in \cite{KorMatHau}). By the simple fact that $\gcd(x_i,x_j)\mathrm{lcm}(x_i,x_j)=x_ix_j$ we have
\[
[S]=\mathrm{diag}(x_1,x_2,\ldots,x_n)\left(\frac{1}{\gcd(x_i,x_j)}\right)\mathrm{diag}(x_1,x_2,\ldots,x_n).
\]
In order to continue and factorize the matrix $\left(\frac{1}{\gcd(x_i,x_j)}\right)$ even further we adopt a technique introduced by Rajarama Bhat \cite{Bhat} in 1991. We apply Möbius inversion and define the function $\Psi_{S,\frac{1}{N}}$ on $S$ recursively as
\begin{equation*}
\Psi_{S,\frac{1}{N}}(x_1)=\frac{1}{x_1},\quad\Psi_{S,\frac{1}{N}}(x_i)=\frac{1}{x_i}-\sum_{x_j\prec x_i}\Psi_{S,\frac{1}{N}}(x_j)
\end{equation*}
or equivalently
\begin{equation}\label{eq:psi}
\Psi_{S,\frac{1}{N}}(x_i)=\sum_{x_j\preceq x_i}\frac{\mu_S(x_j,x_i)}{x_j}.
\end{equation}

Now the matrix $\left(\frac{1}{\gcd(x_i,x_j)}\right)$ may be written as
\[
\left(\frac{1}{\gcd(x_i,x_j)}\right)=E\,\mathrm{diag}(\Psi_{S,\frac{1}{N}}(x_1),\Psi_{S,\frac{1}{N}}(x_2),\ldots,\Psi_{S,\frac{1}{N}}(x_n))\,E^T,
\]
where $E=(e_{ij})$ is the $0,1$ incidence matrix of the set $S$ with
\[
e_{ij}=
\begin{cases}
1 &\text{if }x_j\preceq x_i,\\
0 &\text{otherwise.}
\end{cases}
\]
Putting all together we obtain
\begin{equation}\label{eq:factorization}
[S]=\Delta E\Lambda E^T\Delta=(\Delta E)\Lambda(\Delta E)^T,
\end{equation}
where $\Delta=\mathrm{diag}(x_1,x_2,\ldots,x_n)$ and $\Lambda=\mathrm{diag}(\Psi_{S,\frac{1}{N}}(x_1),\Psi_{S,\frac{1}{N}}(x_2),\ldots\Psi_{S,\frac{1}{N}}(x_n)).$

Since the matrix $\Delta E$ is clearly invertible (triangular matrix with nonzero diagonal elements), the matrix $[S]$ is invertible if and only if the matrix $\Lambda$ is invertible. Moreover, the invertibility of $\Lambda$ can be determined easily since 
\[
\det\Lambda=\Psi_{S,\frac{1}{N}}(x_1)\Psi_{S,\frac{1}{N}}(x_2)\cdots\Psi_{S,\frac{1}{N}}(x_n).
\]
From this we easily obtain the following fundamental result.

\begin{proposition}\label{th:invertibility}
If the set $S$ is GCD closed, then the LCM matrix $[S]$ is invertible if and only if $\Psi_{S,\frac{1}{N}}(x_i)\neq0$ for all $i=1,2,\ldots,n.$ 
\end{proposition}

The next theorem gives us a method to at least narrow down the possible zero terms $\Psi_{S,\frac{1}{N}}(x_i)$ just by looking at the semilattice structure $(S,|).$

\begin{theorem}\label{th:i-tapaus}
Let $S$ be a GCD closed set. If the element $x_i\in S$ generates a double-chain set in $S$, then $\Psi_{S,\frac{1}{N}}(x_i)\neq 0.$
\end{theorem}

\begin{proof}
We are going to show that since the element $x_i$ generates a double-chain set in $S$, in every nontrivial case the positive terms in the sum of Equation \eqref{eq:psi} cancel out all the negative terms. If $C_S(x_i)=\emptyset$, then $x_i$ must be the smallest element in $S$ (i.e. $i=1$) and we have
\[
\Psi_{S,\frac{1}{N}}(x_i)=\frac{1}{x_i}>0.
\]
Next, if $x_i$ covers only one element $x_k\in S$, then we obtain
\[
\Psi_{S,\frac{1}{N}}(x_i)=\frac{1}{x_i}-\frac{1}{x_k}<0,
\]
since $0<x_k<x_i$.

Now we may assume that there are at least two elements in $C_S(x_i)$. This means that there is at least one element in the set $A_i\cup B_i=\mathrm{meetcl}(C_S(x_i))\setminus C_S(x_i)$. In this case the terms in Equation \eqref{eq:psi} may be rearranged so that all the terms $\frac{\mu_S(x_k,x_i)}{x_k}$ with $x_k\in C_S(x_i)$ and $x_j\lessdot x_k$ are put together with the term $\frac{\mu_S(x_j,x_i)}{x_j}$, see Figure \ref{fig:illustration3}. We obtain 
\begin{align}\label{eq:ositettu psi}
\Psi_{S,\frac{1}{N}}(x_i)&\geq\frac{\mu_S(x_i,x_i)}{x_i}+\sum_{x_j\in A_i\cup B_i}\left(\frac{\mu_S(x_j,x_i)}{x_j}+\sum_{\substack{x_k\in C_S(x_i)\\x_j\lessdot x_k\lessdot x_i}}\frac{\mu_S(x_k,x_i)}{x_k}\right)\notag\\ 
&=\frac{1}{x_i}+\sum_{x_j\in A_i\cup B_i}\left(\frac{\mu_S(x_j,x_i)}{x_j}+\sum_{\substack{x_k\in C_S(x_i)\\x_j\lessdot x_k\lessdot x_i}}\frac{-1}{x_k}\right),\notag\\
\end{align}
where equality holds if no element in $C_S(x_i)$ is attached to both of the chains (if $x_k$ is attached to both chains, then the terms differ by $\frac{1}{x_k}$). It should be pointed out that in the above sum $x_j\lessdot x_k\lessdot x_i$ means that these elements cover one another in the poset $\mathrm{meetcl}(C_S(x_i)),$ but not necessarily in $S$ (also note that some of the above sums might be empty depending on the number of elements in $A_i$ and $B_i$). We are going to show that under these circumstances every summand $$\frac{\mu_S(x_j,x_i)}{x_j}+\sum_{\substack{x_k\in C_S(x_i)\\x_j\lessdot x_k\lessdot x_i}}\frac{-1}{x_k}$$
is in fact positive, which implies that $\Psi_{S,\frac{1}{N}}(x_i)>0$ (and particularly that $\Psi_{S,\frac{1}{N}}(x_i)\neq 0$).

In the case when $x_j$ is a maximal element in $A_i\cup B_i$ we have $\eta(x_j)\geq 2$ and $\mu_S(x_j,x_i)=\eta(x_j)-1$ by Theorem \ref{th:mobius}. In addition, for every $x_k\in C_S(x_i)$ with $x_j\lessdot x_k\lessdot x_i$ we have $x_k=a_kx_j$, where the $a_k$'s are distinct and $a_k\geq 2$ for all $k=1,\ldots,\eta(x_j).$ Since $\mu_S(x_j,x_i)\geq \eta(x_j)-1$, we have
\begin{align*}
&\frac{\mu_S(x_j,x_i)}{x_j}+\sum_{\substack{x_k\in C_S(x_i)\\x_j\lessdot x_k\lessdot x_i}}\frac{-1}{x_k}=\frac{\eta(x_j)}{x_j}-\frac{1}{x_j}-\underbrace{\left(\sum_{\substack{x_k\in C_S(x_i)\\x_j\lessdot x_k\lessdot x_i}}\frac{1}{x_k}\right)}_{\eta(x_j)\text{\ terms}}\\
&=\frac{1}{x_j}\left(\eta(x_j)-1-\sum_{\substack{x_k\in C_S(x_i)\\x_j\lessdot x_k\lessdot x_i}}\frac{x_j}{x_k}\right)\geq \frac{1}{x_j}\left(\eta(x_j)-1-\left(\frac{1}{2}+\frac{1}{3}+\cdots+\frac{1}{\eta(x_j)+1}\right)\right)\\
&\geq \frac{1}{x_j}\left(\eta(x_j)-\frac{3}{2}-\frac{\eta(x_j)-1}{3}\right)=\frac{4\eta(x_j)-7}{6x_j}\geq\frac{4\cdot2-7}{6x_j}=\frac{1}{6x_j}>0.
\end{align*}

For all the non-maximal elements $x_j\in A_i\cup B_i$ we have $\mu_S(x_j,x_i)\geq \eta(x_j)$ by Theorem \ref{th:mobius} and therefore
\begin{align*}
&\frac{\mu_S(x_j,x_i)}{x_j}+\sum_{\substack{x_k\in C_S(x_i)\\x_j\lessdot x_k\lessdot x_i}}\frac{-1}{x_k}\geq \frac{\eta(x_j)}{x_j}+\underbrace{\sum_{\substack{x_k\in C_S(x_i)\\x_j\lessdot x_k\lessdot x_i}}\frac{-1}{x_k}}_{\eta(x_j)\text{\ terms}}=\sum_{\substack{x_k\in C_S(x_i)\\x_j\lessdot x_k\lessdot x_i}}\underbrace{\left(\frac{1}{x_j}-\frac{1}{x_k}\right)}_{>0}>0.
\end{align*}
Thus we have shown that $\Psi_{S,\frac{1}{N}}(x_i)>0$ when there are at least two elements in the set $C_S(x_i)$ and therefore our proof is complete.
\end{proof}

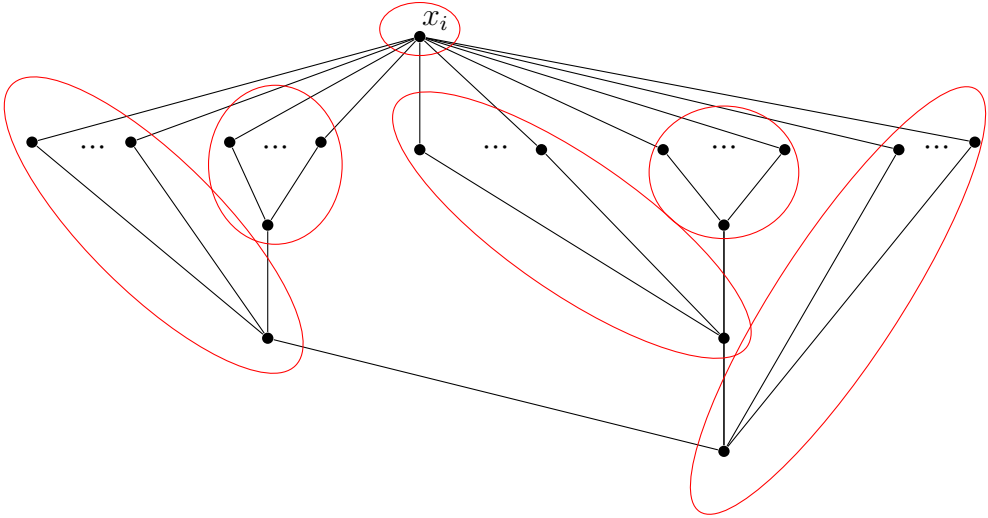
\begin{figure}[ht]
\tikzstyle{every node}=[fill,circle,inner sep=1.5pt]
\centering
\begin{tikzpicture}
\node (a) at (1.9,4.1)  {};
  \node[draw=none, fill=none] (b) at (2.7,4) {$\cdots$};
  \node (c) at (3.2,4.1)  {};
  \node (d) at (4.5,4.1)  {};
  \node[draw=none, fill=none] (e) at (5.1,4) {$\cdots$};
  \node (f) at (5.7,4.1) {};
  \node (h) at (7,4) {};
  \node[draw=none, fill=none] (i) at (8,4) {$\cdots$};
  \node (j) at (8.6,4) {};
  \node (k) at (10.2,4) {};
  \node[draw=none, fill=none] (l) at (11,4) {$\cdots$};
  \node (m) at (11.8,4) {};
  \node (n) at (13.3,4) {};
  \node[draw=none, fill=none] (o) at (13.8,4) {$\cdots$};
  \node (p) at (14.3,4.1) {};
  \node (q) at (11,3) {};
  \node (r) at (11,1.5) {};
  \node (s) at (11,0) {};
  \node (t) at (5,3) {};
	\node (u) at (7,5.5) {};
	\node (v) at (5,1.5) {};
	\node[draw=none, fill=none] [above right] at (u) {$x_i$};
	\path (a) edge (u);
	\path (c) edge (u);
	\path (d) edge (u);
	\path (f) edge (u);
	\path (h) edge (u);
	\path (j) edge (u);
	\path (k) edge (u);
	\path (m) edge (u);
	\path (n) edge (u);
	\path (p) edge (u);
	\path (s) edge (q);
	\path (s) edge (v);
	\path (a) edge (v);
	\path (c) edge (v);
	\path (s) edge (q);
	\path (t) edge (d);
	\path (f) edge (t);
	\path (k) edge (q);
	\path (m) edge (q);
	\path (t) edge (v);
	\path (h) edge (r);
	\path (j) edge (r);
	\path (s) edge (n);
	\path (s) edge (p);
	\draw[red] (11,3.7) ellipse (28 pt and 25 pt);
	\draw[rotate around={57:(12.5,2)}, red] (12.5,2) ellipse (95 pt and 23 pt);
	\draw[rotate around={135:(3.5,3)}, red] (3.5,3) ellipse (75 pt and 25 pt);
	\draw[red] (5.1,3.8) ellipse (25 pt and 30 pt);
	\draw[rotate around={145:(9,3)}, red] (9,3) ellipse (80 pt and 25 pt);
	\draw[red] (7,5.6) ellipse (15 pt and 10 pt);
\end{tikzpicture}
\caption{Illustration on how the sum in Equation \eqref{eq:ositettu psi} is partitioned in the case when no element of $C_S(x_i)$ is attached to both of the chains.}
\label{fig:illustration3}
\end{figure}

The following corollary is an immediate consequence of Theorem \ref{th:i-tapaus} and Proposition \ref{th:invertibility}.

\begin{corollary}\label{cor:1-n}
If the set $S$ is GCD closed and every element $x_i\in S$ generates a double-chain set in $S$, then the LCM matrix $[S]$ is invertible.
\end{corollary}

\begin{proof}
Theorem \ref{th:i-tapaus} implies that $\Psi_{S,\frac{1}{N}}(x_i)\neq 0$ for all $x_i\in S$. Therefore we must have
\[
\det [S]=\Psi_{S,\frac{1}{N}}(x_1)\Psi_{S,\frac{1}{N}}(x_2)\cdots\Psi_{S,\frac{1}{N}}(x_n)\neq0.
\]
\end{proof}

Corollary \ref{cor:1-n} itself has interesting consequences as well. It can easily be used to show one more time that the Bourque-Ligh conjecture holds for all GCD closed sets with at most $7$ elements. Moreover, at the same time we are able to show that among all the meet semilattice structures with $8$ elements there is only one possible exception such that GCD closed sets isomorphic to it may not satisfy the Bourque-Ligh conjecture. In \cite{KorMatHau} and in \cite{MatHau3} this same information was obtained by generating all possible semilattice structures and by going through dozens of special cases.

\begin{theorem}\label{th:B-L-conjecture}
The LCM matrix of any GCD closed set with at most $7$ elements is invertible. Moreover, if there are $8$ elements in the set $S$ and the matrix $[S]$ is not invertible, then $(S,|)$ is isomorphic to the cube semilattice (see Figure \ref{fig:ex1} (e)).
\end{theorem}

\begin{proof}
Suppose that $S$ is a GCD closed set such that the matrix $[S]$ is not invertible. Corollary \ref{cor:1-n} implies that there is at least one element $x_i\in S$ such that the element $x_i$ does not generate a double-chain set in $S$. This means that there must be at least three elements in the set $C_S(x_i)$ (otherwise $x_i$ would generate a double-chain set in $S$). And since the set $\mathrm{meetcl}(C_S(x_i))\setminus C_S(x_i)$ cannot be presented as a union of two distinct chains, there must be at least three incomparable elements $x_p$, $x_q$ and $x_r$ in the set $\mathrm{meetcl}(C_S(x_i))\setminus C_S(x_i)$. Moreover, all the elements $x_p$, $x_q$ and $x_r$ are of the form $x_{i_1}\wedge x_{i_2}\wedge\cdots\wedge x_{i_k}$, where $x_{i_1},x_{i_2},\ldots,x_{i_k}\in C_S(x_i)$ and $k\geq 2$. This means that every one of the elements $x_p$, $x_q$ and $x_r$ precedes at least two of the elements in $C_S(x_i)$. And finally, since the set $\mathrm{meetcl}(C_S(x_i))\setminus C_S(x_i)\subseteq S$ is meet closed, also the element $x_p\wedge x_q\wedge x_r$ must belong to this set. All in all, there must be at least $8$ elements in the set $S$, and the cube semilattice presented in Figure \ref{fig:ex1} (e) is the only $8$-element meet semilattice that meets all these criteria.
\end{proof}

\begin{remark}
It should be noted that the method used in Theorem \ref{th:i-tapaus} may be applicable even if the element $x_i$ does not generate a double-chain set in $S$. For example, if $S$ is a GCD closed set with its structure is presented in Figure \ref{fig:non-AB-set}, the matrix $[S]$ is invertible although the maximum element $x_{12}$ does not generate a double-chain set in $S$. The explanation to this is that the method of calculating the Möbius function values introduced in Theorem \ref{th:mobius} actually works also in the case of this non-double-chain set, and therefore the arguments used in Theorem \ref{th:i-tapaus} will also work.
\end{remark}

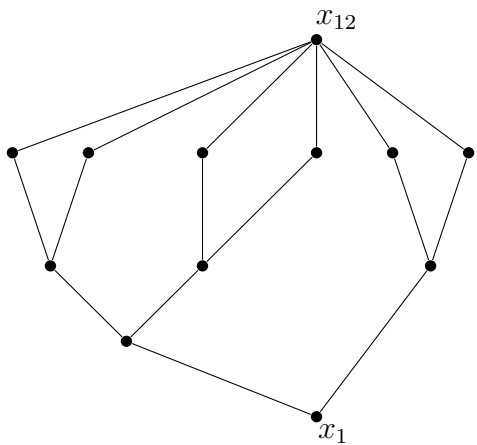
\begin{figure}[ht]
\tikzstyle{every node}=[fill,circle,inner sep=1.5pt]
\centering
\begin{tikzpicture}
\node (a) at (0,6)  {};
  \node (b) at (-4,4.5)  {};
  \node (c) at (-3,4.5)  {};
  \node (d) at (-1.5,4.5) {};
  \node (e) at (0,4.5) {};
  \node (f) at (1,4.5) {};
  \node (g) at (2,4.5) {};
  \node (h) at (-3.5,3) {};
  \node (i) at (-1.5,3) {};
  \node (j) at (1.5,3) {};
  \node (k) at (-2.5,2) {};
  \node (l) at (0,1) {};
	\node[draw=none, fill=none] [above right] at (a) {$x_{12}$};
	\node[draw=none, fill=none] [below right] at (l) {$x_1$};
	\path (a) edge (b);
	\path (a) edge (c);
	\path (a) edge (d);
	\path (a) edge (e);
	\path (a) edge (f);
	\path (a) edge (g);
	\path (b) edge (h);
	\path (c) edge (h);
	\path (d) edge (i);
	\path (e) edge (i);
	\path (f) edge (j);
	\path (g) edge (j);
	\path (h) edge (k);
	\path (i) edge (k);
	\path (k) edge (l);
	\path (j) edge (l);
\end{tikzpicture}
\caption{An example of a non-double-chain set $S$ whose LCM matrix is nevertheless invertible by the argument used in Theorem \ref{th:i-tapaus}.}
\label{fig:non-AB-set}
\end{figure}

It turns out that our theorems can rather easily be applied in the study of LCM matrices of so-called $r$-fold GCD closed sets $S$. First, however, we need to define this concept. 

\begin{definition}\cite[Definition 2.1]{Hong06}\label{def:r-fold}
Let $T$ be a set of $n$ distinct positive integers and $r\in [1,n-1]$ be an integer. We say that $T$ is a $0$-\emph{fold GCD closed set} if $T$ is GCD closed. We say that $T$ is an $r$-\emph{fold GCD closed set} if there is a divisor chain $R\subseteq T$ with $|R|=r$ such that $\max(R)\,|\,\min(T\setminus R)$ and the set $T\setminus R$ is GCD closed.  
\end{definition}

We are now ready to give a novel lattice-theoretic proof for the following result whose first appearance was in Shaofang Hong's PhD thesis in 1998.

\begin{corollary}\cite[Theorem 2.1 (ii)]{Hong06}\label{th:r-fold}
If $n\geq 8$ and $S=\{x_1,x_2,\ldots,x_n\}$ is an $(n-7)$-fold GCD closed set, then the LCM matrix $[S]$ is nonsingular.
\end{corollary}

\begin{proof}
Suppose that $S=\{x_1,x_2,\ldots,x_n\}$ is an $(n-7)$-fold GCD closed set, where $x_1<x_2<\ldots<x_n$. We only need to show that every element $x_i\in S$ generates a double-chain set in $S$, and then the claim follows directly from Corollary \ref{cor:1-n}. If $1\leq i\leq n-7$, then $x_i$ is a member of the divisor chain and covers at most one element in the semilattice $(S,|)$ and clearly $x_i$ generates a double-chain set in$S$. The case when $x_i=x_{n-6}$ is similar, since $C_S(x_{n-6})=\{x_{n-7}\}$. If $n-5\leq i\leq n$, then $\mathrm{meetcl}(C_S(x_i))\subseteq\{x_{n-6},x_{n-5},\ldots,x_{n-1}\}$. Since there are at most $6$ elements in this semilattice we may continue as in the proof of Theorem \ref{th:B-L-conjecture} and show that there cannot be three incomparable elements in the set $C_S(x_i)$. Thus also in this case $x_i$ generates a double-chain set in $S$.
\end{proof}

If $S$ is a GCD closed $\wedge$-tree set, then every element $x_i\in S$ trivially generates a double-chain set in $S$. We directly obtain the following corollary, which is another immediate consequence of Corollary \ref{cor:1-n}.

\begin{corollary}\label{cor:tree set inv}
The matrix $[S]$ is invertible for all GCD closed $\wedge$-tree sets $S$.
\end{corollary}

Even Corollary \ref{cor:tree set inv} itself has a few interesting consequences. In 1998 Hong applied his number-theoretic method and proved the results listed in Corollary \ref{cor:hong98}. Since in each case the set $S$ happens to be a GCD closed $\wedge$-tree set, each claim follows directly from Corollary \ref{cor:tree set inv}.

\begin{corollary}\cite[Corollary 3.2, Corollary 3.3 and Corollary 3.6]{Hong98}\label{cor:hong98}
Let $$S=\{x_1,x_2,\ldots,x_n\}$$ be a GCD closed set with $x_1<x_2<\cdots<x_n$.
\begin{enumerate}
\item If $x_i\,|\,x_j$ for all $1\leq i<j\leq n$, then the LCM matrix $[S]$ is nonsingular.
\item If the numbers $\frac{x_2}{x_1},\frac{x_3}{x_1},\ldots,\frac{x_n}{x_1}$ are pairwise relatively prime, then the LCM matrix [S] is nonsingular. 
\item If for each $x_k$, where $3\leq k\leq n$, there exists at most one $x_{i_k}\in S$ such that $2\leq i_k\leq k-1$ and $x_{i_k}\,|\,x_k$, then the LCM matrix $[S]$ is nonsingular. 
\end{enumerate}
\end{corollary}

\begin{remark}
One may also be interested in the invertibility of the so-called power LCM matrix of a GCD closed set $S$ with $(\mathrm{lcm}(x_i,x_j))^\alpha$ as its $ij$ entry, where $\alpha$ is an arbitrary real exponent. A slight modification to the proof of Theorem \ref{th:i-tapaus} would enable us to show that the power LCM matrix $[(\mathrm{lcm}(x_i,x_j))^\alpha]$ of a double-chain set $S$ is invertible for all $\alpha\geq 1$. However, we do not present the details for the sake of brevity.
\end{remark}

\section{The inertia of LCM matrices on GCD closed double-chain sets}\label{sect:inertia}

In \cite{Alt16} the authors conclude their work by presenting the following two open problems about the inertia of LCM matrices:
\begin{enumerate}
\item How many of the eigenvalues of the classical LCM matrix $[\mathrm{lcm}(i,j)]$ of the set $S=\{1,2,\ldots,n\}$ are positive?
\item How many of the eigenvalues of the general LCM matrix $[\mathrm{lcm}(x_i,x_j)]$ of the set $S=\{x_1,x_2,\ldots,x_n\}$ are positive?
\end{enumerate}
Even the question about the inertia of the classical LCM matrix of the set $S=\{1,2,\ldots,n\}$ is far from trivial. However, it turns out that by the work done earlier in this paper in some cases it is possible to completely determine the inertia of the matrix $[S]$ simply by looking at the semilattice structure of $(S,|)$. 

\begin{theorem}\label{th:inertia}
Suppose that $S$ is a GCD closed set where every element $x_i\in S$ generates a double-chain set in $S$. The inertia of the LCM matrix $[S]$ is the triple $(i_+([S]),i_-([S]),i_0([S])),$ where $i_0([S])=0,$ $$i_-([S])=\left|\{x_k\in S\,\big|\,|C_S(x_k)|=1\}\right|$$ and $$i_+([S])=n-i_-([S])=\left|\{x_k\in S\,\big|\,|C_S(x_k)|\neq1\right|.$$
\end{theorem}

\begin{proof}
From Equation \eqref{eq:factorization} we see that the LCM matrix $[S]$ and the diagonal matrix $\Lambda$ are congruent. By Sylvester's law of inertia (see e.g. \cite[Theorem 4.5.8]{Horn}), these matrices have the same inertia. In the proof of Theorem \ref{th:i-tapaus} it was shown that if $x_i$ generates a double-chain set in $S$, then $$|C_S(x_i)|\neq 1\Rightarrow\Psi_{S,\frac{1}{N}}(x_i)>0$$ and $$|C_S(x_i)|= 1\Rightarrow\Psi_{S,\frac{1}{N}}(x_i)<0.$$ The claim follows from this.
\end{proof}

\begin{example}\label{ex2}
Let us consider the inertia of the classical LCM matrix of the set $S=\{1,2,\ldots,n\}$. It is easy to see that some element $i\in S$ covers exactly one element with respect to divisor relation $|$ if and only if $i$ is a prime power greater than one (if $p\,|\,i$ and $q\,|\,i$, where $p$ and $q$ are two distinct primes, then $\frac{i}{p}\lessdot i$ and $\frac{i}{q}\lessdot i$). For small $n$ this observation together with Theorem \ref{th:inertia} gives a simple way of determine the inertia of the matrix $[S]$. For example, if $n=12$ we have $i_-([S])=8$ and $i_+([S])=4$. It would also be possible to draw the same conclusion from the Hasse diagram of $(S,|)$, see Figure \ref{fig:ex2}. However, if $n\geq 30$, then every element in $(S,|)$ does not generate a double-chain set in $S$ and for such elements $x_i$ the value $\Psi_{S,\frac{1}{N}}(x_i)$ needs to be calculated from Equation \eqref{eq:psi} (see Example \ref{ex4}).
\end{example}

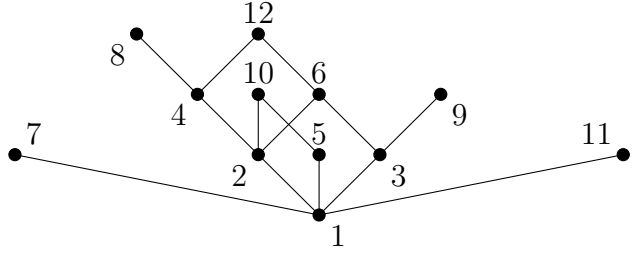
\begin{figure}[htb!]
\centering
{
\begin{tikzpicture}[scale=0.8]
\draw (3,0)--(0,3);
\draw [fill] (1,2) circle [radius=0.1];
\draw [fill] (2,1) circle [radius=0.1];
\draw [fill] (3,0) circle [radius=0.1];
\draw [fill] (2,2) circle [radius=0.1];
\draw [fill] (3,1) circle [radius=0.1];
\draw [fill] (2,3) circle [radius=0.1];
\draw [fill] (3,2) circle [radius=0.1];
\draw [fill] (4,1) circle [radius=0.1];
\draw [fill] (5,2) circle [radius=0.1];
\draw [fill] (-2,1) circle [radius=0.1];
\draw [fill] (0,3) circle [radius=0.1];
\draw [fill] (8,1) circle [radius=0.1];
\draw (3,0)--(5,2);
\draw (3,0)--(3,1);
\draw (2,1)--(2,2);
\draw (3,1)--(2,2);
\draw (3,0)--(-2,1);
\draw (2,1)--(3,2);
\draw (1,2)--(2,3);
\draw (4,1)--(2,3);
\draw (3,0)--(8,1);
\node [below right] at (3,0) {$1$};
\node [below left] at (2,1) {$2$};
\node [below left] at (1,2) {$4$};
\node [below left] at (0,3) {$8$};
\node [below right] at (4,1) {$3$};
\node [below right] at (5,2) {$9$};
\node [above right] at (-2,1) {$7$};
\node [above] at (3,1) {$5$};
\node [above] at (2,2) {$10$};
\node [above] at (3,2) {$6$};
\node [above] at (2,3) {$12$};
\node [above left] at (8,1) {$11$};
\end{tikzpicture}
}
\caption{The Hasse diagram of the semilattice $(\{1,2,\ldots,12\},|)$ in Example \ref{ex2}.}\label{fig:ex2}
\end{figure}

\begin{example}\label{ex3}
Next we take a new look at some of the semilattice structures in Example \ref{ex1}. If $S$ is a GCD closed set whose Hasse diagram is isomorphic to the one presented in
\begin{itemize}
\item Figure \ref{fig:ex1} (a), then $i_-([S])=n-1$ and $i_+([S])=1$ (the same is true for any $A$-set $S$),
\item Figure \ref{fig:ex1} (b) or (c), then $i_-([S])=n-2$ and $i_+([S])=2$,
\item Figure \ref{fig:ex1} (d), then $i_-([S])=n-3$ and $i_+([S])=3$.
\end{itemize}
\end{example}

We are now in a position to give a more through consideration of GCD closed semilattices isomorphic to the cube semilattice (see Figure \ref{fig:ex1} (e)). We shall see that if some element $x_i$ does not generate a double-chain set in $S$, then there is no shortcut in determining the inertia of the matrix $[S]$ but the only choice is to calculate the respective value of $\Psi_{S,\frac{1}{N}}(x_i)$.

\begin{example}\label{ex4}
Suppose that the set $S=\{x_1,x_2,\ldots,x_8\}$ is GCD closed and isomorphic to the cube semilattice in Figure \ref{ex1} (e). Since every other element except for the greatest element $x_8$ generates a double-chain set in $S$, we may deduce that at least four of the eigenvalues of the matrix $[S]$ are positive and at least three of them are negative. However, when it comes to the last eigenvalue, anything is possible:
\begin{itemize}
\item Suppose that $S=\{1,2,3,5,6,10,15,30\}$. Then we have
\[
\Psi_{S,\frac{1}{N}}(x_8)=\frac{1}{30}-\frac{1}{15}-\frac{1}{10}-\frac{1}{6}+\frac{1}{5}+\frac{1}{3}+\frac{1}{2}-\frac{1}{1}=-\frac{4}{15}<0,
\]
and therefore $i_-([S])=4$, $i_+([S])=4$ and $i_0([S])=0$.
\item Suppose that $S=\{1,2,3,5,70,78,255,46410\}$. Then we have
\[
\Psi_{S,\frac{1}{N}}(x_8)=\frac{1}{464100}-\frac{1}{255}-\frac{1}{78}-\frac{1}{70}+\frac{1}{5}+\frac{1}{3}+\frac{1}{2}-\frac{1}{1}>\frac{2}{1000}>0,
\]
and therefore $i_-([S])=3$, $i_+([S])=5$ and $i_0([S])=0$.
\item Suppose that $S=\{1,2,3,5,66,70,255,39270\}$. Then we have
\[
\Psi_{S,\frac{1}{N}}(x_8)=\frac{1}{39270}-\frac{1}{255}-\frac{1}{70}-\frac{1}{66}+\frac{1}{5}+\frac{1}{3}+\frac{1}{2}-\frac{1}{1}=0,
\]
and therefore $i_-([S])=3$, $i_+([S])=4$ and $i_0([S])=1$.
\end{itemize}
\end{example}

So far we have seen the usefulness of our results in determining the inertia of the LCM matrix of a certain individual GCD closed set, but is it possible to state something about the inertia that holds for all LCM matrices of GCD closed sets of order $n$ in general?  The following theorem gives some answers to this question. 

\begin{theorem}\label{th:inertia bounds}
Suppose that the set $S$ is GCD closed with $n\geq 3$ elements. Then $i_+([S])\geq 1$ and
\[
2\leq i_-([S])\leq n-1.
\]
\end{theorem}

\begin{proof}
Without the loss of generality we may assume that the elements of $S$ are indexed in a non-decreasing order (in other words, the condition $x_i\,|\,x_j\Rightarrow i\leq j$ is satisfied). Since the minimum element $x_1$ trivially generates a double-chain set in $S$ and $|C_S(x_1)|=0$, we must have $i_+([S])\geq 1$.  The lower bound for $i_-([S])$ follows from the fact that the elements $x_2$ and $x_3$ also generate a double-chain set in $S$ and we have $|C_S(x_2)|=|C_S(x_3)|=1$. The upper bound for $i_-([S])$ is again trivial, since we have $i_+([S])\geq 1$.
\end{proof}

\begin{remark}
It should be noted that the bounds for $i_-([S])$ given in Theorem \ref{th:inertia bounds} cannot be improved without making further assumptions. For example, if $S=\{1,p,q,pq\}$, where $p$ and $q$ are distinct prime numbers, we have $i_-([S])=i_+([S])=2$. On the other hand, if the set $S$ is a $\wedge$-tree set where every element apart from $x_1$ covers at most one element, we have $i_-([S])= n-1$.
\end{remark}

Finding a sharp upper bound for $i_+([S])$ when $n=|S|$ is fixed appears to be much more difficult than it was with $i_-([S])$, and it is even more difficult is to find semilattice structures that maximize $i_+([S])$ (see Problem \ref{problem} at the end of this paper). Increasing the number of positive eigenvalues tends to (at least to some extent) increase the number of negative eigenvalues as well. However, the following examples show that there exist semilattice structures for which the ratio $\frac{i_+([S])}{n}$ is extremely close to $1$.

\begin{example}\label{ex:pq}
Let $p$ and $q$ be distinct prime numbers and let $m>1$ be a positive integer. We define
\[
S=\{p^kq^l\,\big|\,0\leq k,l\leq m-1\},
\]
where the structure of $(S,|)$ is being illustrated in Figure \ref{fig:pq-ex} (and $m^2=n$). The set $S$ is clearly a double-chain set (every element covers at most two elements). Moreover, there are exactly $2(m-1)$ prime power elements $p^k$ and $q^l$ ($k,l>0$) that cover exactly one element in $S$. Thus we have $i_-([S])=2m-2$, $i_+([S])=(m-1)^2+1$ and
\[
\frac{i_+([S])}{n}=\frac{m^2-2m+2}{m^2}\to1
\]
as $m\to\infty$.
\end{example}

\begin{figure}[ht]
\tikzstyle{every node}=[fill,circle,inner sep=1.5pt,scale=0.9]
\centering
\begin{tikzpicture}
\node (a) at (0,0)  {};
  \node (b) at (-1,1)  {};
  \node (c) at (-2,2)  {};
  \node (d) at (-3,3) {};
  \node (e) at (-4,4) {};
  \node (f) at (1,1) {};
  \node (g) at (2,2) {};
  \node (h) at (3,3) {};
  \node (i) at (4,4) {};
  \node (j) at (0,2) {};
  \node (k) at (-1,3) {};
  \node (l) at (-2,4) {};
	\node (m) at (-3,5) {};
  \node (n) at (1,3) {};
  \node (o) at (0,4) {};
  \node (p) at (-1,5) {};
  \node (q) at (-2,6) {};
	\node (r) at (2,4) {};
  \node (s) at (1,5) {};
  \node (t) at (0,6) {};
  \node (u) at (-1,7) {};
	\node (v) at (3,5) {};
  \node (w) at (2,6) {};
  \node (x) at (1,7) {};
  \node (y) at (0,8) {};
	\node[draw=none, fill=none] [above right] at (y) {$p^{m-1}q^{m-1}$};
	\node[draw=none, fill=none] [below right] at (a) {$1$};
	\node[draw=none, fill=none] [below right] at (f) {$q$};
	\node[draw=none, fill=none] [below left] at (b) {$p$};
	\node[draw=none, fill=none] [right] at (j) {\ $pq$};
	\node[draw=none, fill=none] [below right] at (i) {$q^{m-1}$};
	\node[draw=none, fill=none] [above right] at (v) {$pq^{m-1}$};
	\node[draw=none, fill=none] [below left] at (e) {$p^{m-1}$};
	\node[draw=none, fill=none] [above left] at (m) {$p^{m-1}q$};
	\node[draw=none, fill=none] [above right] at (x) {$p^{m-2}q^{m-1}$};
	\node[draw=none, fill=none] [above left] at (u) {$p^{m-1}q^{m-2}$};
	\path (a) edge (e);
	\path (a) edge (i);
	\path (f) edge (m);
	\path (g) edge (q);
	\path (h) edge (u);
	\path (i) edge (y);
	\path (c) edge (w);
	\path (d) edge (x);
	\path (e) edge (y);
	\path (b) edge (v);
	\end{tikzpicture}
\caption{The Hasse diagram of $(S,|)$ defined in Example \ref{ex:pq}.}
\label{fig:pq-ex}
\end{figure} 

\begin{example}\label{ex:pq2}
Let $p_1,\ldots,p_m$ be distinct prime numbers and define $$S=\{1\}\cup\{p_1,\ldots,p_m\}\cup\{p_ip_j\,\big|\,1\leq i<j\leq m\},$$
see Figure \ref{fig:pq2-ex}. It is easy to see that $S$ is both GCD closed and a double-chain set, and thus we have
\[
i_-([S])=m,\quad i_+([S])=1+\binom{m}{2}=\frac{1}{2}(m^2-m+2)
\]
and $n=\frac{1}{2}(m^2+m+2).$

\end{example}

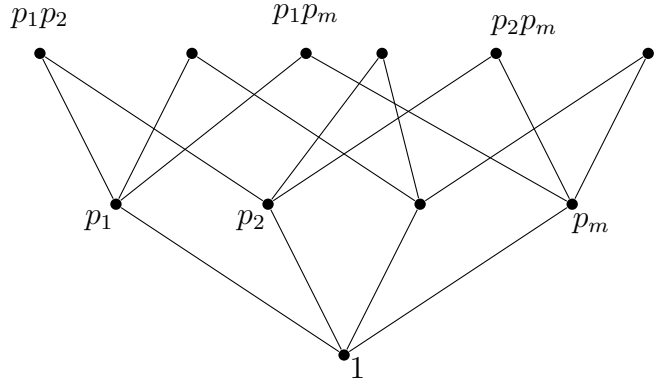
\begin{figure}[ht]
\tikzstyle{every node}=[fill,circle,inner sep=1.5pt]
\centering
\begin{tikzpicture}
\node (a) at (0,0)  {};
  \node (b) at (-3,2)  {};
  \node (c) at (-1,2)  {};
  \node (d) at (1,2) {};
  \node (e) at (3,2) {};
  \node (f) at (-4,4) {};
  \node (g) at (-2,4) {};
  \node (h) at (-0.5,4) {};
  \node (i) at (0.5,4) {};
  \node (j) at (2,4) {};
  \node (k) at (4,4) {};
  \node[draw=none, fill=none] [below left] at (b) {$p_1$};
	\node[draw=none, fill=none] [below right] at (a) {$1$};
	\node[draw=none, fill=none] [below left] at (c) {$p_2$};
	\node[draw=none, fill=none] [below right] at (e) {$p_m$};
	\node[draw=none, fill=none] [above] at (f) {$p_1p_2$};
	\node[draw=none, fill=none] [above] at (h) {$p_1p_m$};
	\node[draw=none, fill=none] [above right] at (j) {$p_2p_m$};
	\path (a) edge (b);
	\path (a) edge (c);
	\path (a) edge (d);
	\path (a) edge (e);
	\path (b) edge (f);
	\path (b) edge (g);
	\path (b) edge (h);
	\path (c) edge (f);
	\path (c) edge (j);
	\path (c) edge (i);
	\path (d) edge (g);
	\path (d) edge (i);
	\path (d) edge (k);
	\path (e) edge (k);
	\path (e) edge (h);
	\path (j) edge (e);
	\end{tikzpicture}
\caption{The Hasse diagram of $(S,|)$ defined in Example \ref{ex:pq2} (for $m=4$).}
\label{fig:pq2-ex}
\end{figure} 

\begin{example}\label{ex:pq3}
Let $p_1,p_2,\ldots,p_m,q,r$ be distinct prime numbers and let us define
\[
S=\bigcup_{i=1}^m \{r^{i-1}q^kp_i^l\,|\,0\leq k,l\leq m-1\}.
\]
The structure of $(S,|)$ is illustrated in Figure \ref{fig:pq3-ex}. It can be shown that $S$ is a GCD closed set as well as a double-chain set and we have
\[
i_-([S])=m^2+m-2,\quad i_+([S])=m^3-m^2-m+2 \quad\text{and}\quad n=m^3.
\]
\end{example}

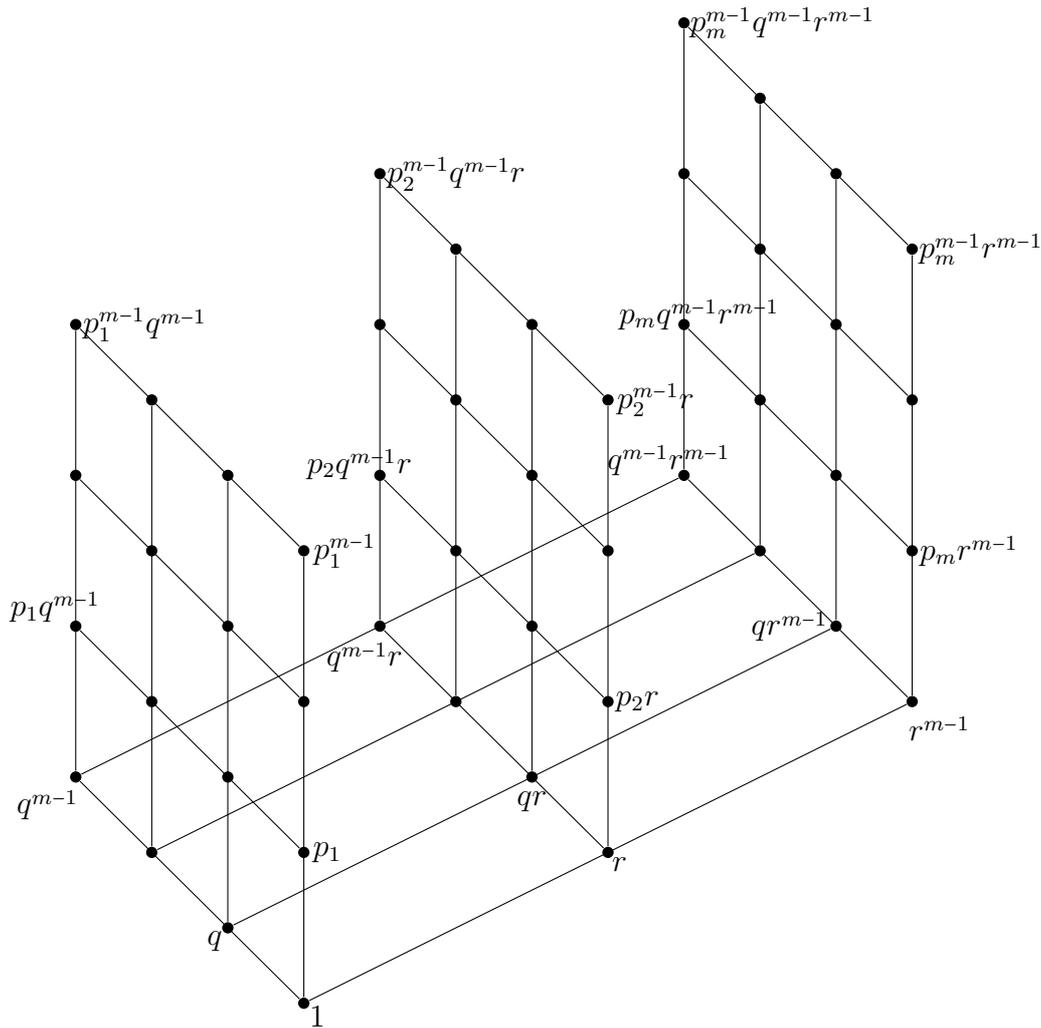
\begin{figure}[htb]
\tikzstyle{every node}=[fill,circle,inner sep=1.5pt]
\centering
\begin{tikzpicture}
\node (a1) at (0,0)  {};
  \node (b1) at (-1,1)  {};
  \node (c1) at (-2,2)  {};
  \node (d1) at (-3,3) {};
  \node (e1) at (0,2) {};
  \node (f1) at (-1,3) {};
  \node (g1) at (-2,4) {};
  \node (h1) at (-3,5) {};
  \node (i1) at (0,4) {};
  \node (j1) at (-1,5) {};
  \node (k1) at (-2,6) {};
	\node (l1) at (-3,7) {};
  \node (m1) at (0,6) {};
  \node (n1) at (-1,7) {};
  \node (o1) at (-2,8) {};
  \node (p1) at (-3,9) {};
	\node (a2) at (4,2)  {};
  \node (b2) at (3,3)  {};
  \node (c2) at (2,4)  {};
  \node (d2) at (1,5) {};
  \node (e2) at (4,4) {};
  \node (f2) at (3,5) {};
  \node (g2) at (2,6) {};
  \node (h2) at (1,7) {};
  \node (i2) at (4,6) {};
  \node (j2) at (3,7) {};
  \node (k2) at (2,8) {};
	\node (l2) at (1,9) {};
  \node (m2) at (4,8) {};
  \node (n2) at (3,9) {};
  \node (o2) at (2,10) {};
  \node (p2) at (1,11) {};
	\node (a3) at (8,4)  {};
  \node (b3) at (7,5)  {};
  \node (c3) at (6,6)  {};
  \node (d3) at (5,7) {};
  \node (e3) at (8,6) {};
  \node (f3) at (7,7) {};
  \node (g3) at (6,8) {};
  \node (h3) at (5,9) {};
  \node (i3) at (8,8) {};
  \node (j3) at (7,9) {};
  \node (k3) at (6,10) {};
	\node (l3) at (5,11) {};
  \node (m3) at (8,10) {};
  \node (n3) at (7,11) {};
  \node (o3) at (6,12) {};
  \node (p3) at (5,13) {};
  \node[draw=none, fill=none] [below left] at (b1) {$q$};
	\node[draw=none, fill=none] [below right] at (a1) {$1$};
	\node[draw=none, fill=none] [right] at (e1) {$p_1$};
	\node[draw=none, fill=none] [right] at (m1) {$p_1^{m-1}$};
	\node[draw=none, fill=none] [right] at (p1) {$p_1^{m-1}q^{m-1}$};
	\node[draw=none, fill=none] [above] at (-3.2,4.5) {$p_1q^{m-1}$\ \ };
	\node[draw=none, fill=none] [below left] at (d1) {$q^{m-1}$};
	\node[draw=none, fill=none] [below] at (b2) {$qr$};
	\node[draw=none, fill=none] [below right] at (a2) {$r$};
	\node[draw=none, fill=none] [right] at (e2) {$p_2r$};
	\node[draw=none, fill=none] [right] at (m2) {$p_2^{m-1}r$};
	\node[draw=none, fill=none] [right] at (p2) {$p_2^{m-1}q^{m-1}r$};
	\node[draw=none, fill=none] [above] at (0.8,6.3) {$p_2q^{m-1}r$\ \ };
	\node[draw=none, fill=none] [below] at (0.8,5.2) {$q^{m-1}r$};
	\node[draw=none, fill=none] [left] at (b3) {$qr^{m-1}$};
	\node[draw=none, fill=none] [below right] at (a3) {$r^{m-1}$};
	\node[draw=none, fill=none] [right] at (e3) {$p_mr^{m-1}$};
	\node[draw=none, fill=none] [right] at (m3) {$p_m^{m-1}r^{m-1}$};
	\node[draw=none, fill=none] [right] at (p3) {$p_m^{m-1}q^{m-1}r^{m-1}$};
	\node[draw=none, fill=none] [above] at (5.2,8) {$p_mq^{m-1}r^{m-1}$};
	\node[draw=none, fill=none] [left] at (5.7,7.2) {$q^{m-1}r^{m-1}$};
	\path (a1) edge (d1);
	\path (e1) edge (h1);
	\path (i1) edge (l1);
	\path (m1) edge (p1);
	\path (d1) edge (p1);
	\path (c1) edge (o1);
	\path (b1) edge (n1);
	\path (a1) edge (m1);
  \path (a2) edge (d2);
	\path (e2) edge (h2);
	\path (i2) edge (l2);
	\path (m2) edge (p2);
	\path (d2) edge (p2);
	\path (c2) edge (o2);
	\path (b2) edge (n2);
	\path (a2) edge (m2);
  \path (a3) edge (d3);
	\path (e3) edge (h3);
	\path (i3) edge (l3);
	\path (m3) edge (p3);
	\path (d3) edge (p3);
	\path (c3) edge (o3);
	\path (b3) edge (n3);
	\path (a3) edge (m3);
	\path (a1) edge (a3);
	\path (b1) edge (b3);
	\path (c1) edge (c3);
	\path (d1) edge (d3);
	\end{tikzpicture}
\caption{The Hasse diagram of $(S,|)$ defined in Example \ref{ex:pq3}.}
\label{fig:pq3-ex}
\end{figure} 

We conclude our study by presenting the following open problem about $i_+([S])$.

\begin{problem}\label{problem}
Let us consider the sequence $(a_n)_{n=1}^\infty$, where
\[
a_n=\max \{i_+([S])\ \big|\ S\subset\Zset^+\text{\ is GCD closed set with\ }|S|=n\}.
\]
For example, we have $a_1=a_2=a_3=1$, $a_4=a_5=2$ and $a_6=3$. Moreover, the previous examples show that $\lim_{n\to\infty}\frac{a_n}{n}=1$. What is the $n$-th term of the sequence?
\end{problem}

It appears to be a rather nontrivial task to solve Problem \ref{problem} completely, but it should not be too difficult to find a good lower bound for $a_n$ by using the ideas presented in the previous examples. In fact, it seems well possible that for certain values of $n$ the structure that maximizes $i_+([S])$ is the one presented in Example \ref{ex:pq2}.

\noindent {\bf Acknowledgements} The authors wish to thank the reviewers for many valuable comments that helped us to improve this paper significantly. The corresponding author also wishes to thank the Finnish Cultural Foundation, Pirkanmaa Regional Fund for supporting the writing of this article.


\begin{thebibliography}{99}

\bibitem{Alt17} E. Altinisik and T. Altintas, A note on the singularity of LCM matrices on GCD-closed sets with 9 elements, {\it Journal of Sciences and Arts} 3 (40): 413--422 (2017).

\bibitem{Alt16} E. Altinisik and S. B\"uy\"ukköse, On bounds for the smallest and the largest eigenvalues of GCD and LCM matrices, {\it Math. Inequal. Appl.} 19 (1): 117--125 (2016).

\bibitem{Alt05} E. Altinisik, B. E. Sagan and N. Tuglu, GCD matrices, posets, and nonintersecting paths, {\it Linear Multilinear Algebra} 53: 75--84 (2005).

\bibitem{BesLigh89} S.~Beslin and S.~Ligh, Greatest common divisor matrices, {\it Linear Algebra Appl.} 174: 69--76 (1989).

\bibitem{Bour92} K.~Bourque and S.~Ligh, On GCD and LCM matrices, {\it Linear Algebra Appl.} 174: 65--74 (1992).


\bibitem{HauToth} P.~Haukkanen and L.~T\'oth, Inertia, positive definiteness and $\ell_p$ norm of GCD and LCM matrices and their unitary analogs, {\it Linear Algebra Appl.} 558: 1--24 (2018).

\bibitem{HauWanSil} P.~Haukkanen, J.~Wang and J.~Sillanp\"a\"a, On Smith's determinant, {\it Linear Algebra Appl.} 258: 251--269 (1997).

\bibitem{Hong98} S.~Hong, On LCM matrices on GCD-closed sets, {\it Southeast Asian Bull. Math.} 22: 381--384 (1998).

\bibitem{Hong99} S.~Hong, On the Bourque-Ligh conjecture of least common multiple matrices, {\it J. Algebra} 218: 216--228 (1999).

\bibitem{Hong04} S.~Hong, Notes on power LCM matrices, {\it Acta Arith.} 111 (2): 165--177 (2004).

\bibitem{Hong06} S.~Hong, Nonsingularity of matrices associated with classes of arithmetical functions on lcm-closed sets, {\it Linear Algebra Appl.} 416: 124--134 (2006).

\bibitem{HongSun} S.~Hong and Q.~Sun, Determinants of matrices associated with incidence functions on posets, {\it Czechoslovak Math. J.} 54 (129): 431--443 (2004).

\bibitem{Horn} R.~A.~Horn and C.~R.~Johnson, {\it Matrix Analysis}, Cambridge University Press, 2nd printing, London, 1985.

\bibitem{IlmKaar} P.~Ilmonen and V.~Kaarnioja, Generalized eigenvalue problems for meet and join matrices on semilattices, {\it Linear Algebra Appl.} 536: 250--273 (2018).

\bibitem{ISmo} I. Korkee, On meet and join matrices on $A$-sets and related sets, {\it Notes Number Theory Discrete Math.} 10 (3): 57--67 (2004).

\bibitem{KorMatHau} I.~Korkee, M.~Mattila and P.~Haukkanen, A lattice-theoretic approach to Bourque-Ligh conjecture, {\it Linear Multilinear Algebra}  DOI: 10.1080/03081087.2018.1494695 (2018).

\bibitem{Li} M.~Li, Notes on Hong's conjectures of real number power LCM matrices, {\it J. Algebra} 315: 654--664 (2007).

\bibitem{MatHau2} M.~Mattila and P.~Haukkanen, On the positive definiteness and eigenvalues of meet and join matrices, {\it Discrete Mathematics} 326: 9--19 (2014).

\bibitem{MatHau4} M.~Mattila and P.~Haukkanen, Studying the various properties of MIN and MAX matrices -- elementary vs. more advanced methods, {\it Spec. Matrices} 2016 (4): 101--109 (2016).

\bibitem{MatHau3} M.~Mattila, P.~Haukkanen and J.~M\"antysalo, Studying the singularity of LCM-type matrices via semilattice structures and their M\"obius functions, {\it J. Combin. Theory Ser. A} 135: 181--200 (2015).

\bibitem{Ovall} J.~S.~Ovall, An analysis of GCD and LCM matrices via the $LDL^T$-factorization, {\it Electron. J. Linear Algebra} 11: 51--58 (2004).

\bibitem{Bhat}  B.~V.~Rajarama Bhat, On greatest common divisor matrices and their applications, {\it Linear Algebra Appl.} 158: 77--97 (1991). 

\bibitem{Shen} Z.~C.~Shen, GCD and LCM power matrices, {Fibonacci Quart.} 34 (4): 290--297 (1996).

\bibitem{Smi} H.~J.~S.~Smith, On the value of a certain arithmetical determinant, {\it Proc.\ London Math.\ Soc.} 7: 208--212 (1875/76).

\end{thebibliography}
\end{document}